\documentclass[a4paper,11pt]{article}
\usepackage{graphicx}
\usepackage{amssymb}
\usepackage{enumerate}
\usepackage{array}
\usepackage{geometry}
\usepackage{fancyhdr}
\usepackage{amsmath}
\usepackage{verbatim}
\usepackage{url}
\usepackage{theoremref}
\usepackage{amscd}
\usepackage{mathrsfs}
\usepackage{color}

\DeclareMathOperator{\stab}{Stab}

\DeclareMathOperator{\depth}{depth}
\DeclareMathOperator{\covol}{covol}

\newtheorem{theorem}{Theorem}[section]
\newtheorem{lemma}[theorem]{Lemma}
\newtheorem{prop}[theorem]{Proposition}

\newtheorem{defn}[theorem]{Definition}
\newtheorem{remark}[theorem]{Remark}

\newenvironment{proof}[1][Proof]{\begin{trivlist}
\item[\hskip \labelsep {\bfseries #1}]}{\end{trivlist}}

\begin{document}

\title{Strong accessibility for hyperbolic groups}
\author{Michael Hill}
\date{\today}
\maketitle

\begin{abstract}
This paper aims to give an account of theorem of Louder and Touikan \cite{LouderTouikan_strong} which shows that many hierarchies consisting of slender JSJ-decompositions are finite. In particular JSJ-hierarchies of $2$--torsion-free hyperbolic groups are always finite.
\end{abstract}

Suppose we are interested in how a group can decompose into a graph of groups. One natural question we can ask is if we can bound the number of vertices of our decomposition given some reasonable hypothesis. The first example of such a result is Grushko's Theorem \cite{Grushko} which implies that the number of vertices of a free splitting of a group (with non-trivial label) is at most the rank of that group. Other similar results include a result from Dunwoody \cite{Dunwoody_fp}, which gives an explicit bound for a finitely presented groups given that the edge groups are finite as well as a generalisation given by Bestvina and Feighn \cite{BestvinaFeighn_small} for when the edge groups are small. (A group is \emph{small} if it fails to act hyperbolically on any tree.) We call such results \emph{accessibility} results. Accessibility need not hold in general; for example Dunwoody \cite{Dunwoody_inaccess} has given an example of a finitely generated group which is not accessible over finite edge groups.

A natural extension of this question is to ask what happens if we recursively split over vertex groups. We naturally get the notion of a \emph{hierarchy} as a rooted tree with a group associated to each vertex, where the immediate descendants of a vertex correspond to the vertex stabilisers of a splitting of its group. We would like to know if our group has finite hierarchies with terminal vertices which have indecomposable groups and so is \emph{strongly accessible} in some sense. Over finite edge groups this question immediately reduces to the regular accessibility question as finite subgroup always fixes a point of a tree; however we run into problems as soon as we begin looking at infinite groups. For example since $F_2 \cong (F_2) *_{\mathbb{Z}}$ (as $F_2 \cong \left\langle a,b,c \mid cac^{-1} = b \right\rangle$) we can easily build an infinite hierarchy for free groups over cyclic edge groups. As such we instead try and show that some particular hierarchy with indecomposable terminal vertices is finite. For example the Haken hierarchy of a $3$--manifold is finite \cite{Haken}.

Delzant and Potyagailo \cite{DelzantPotyagailo_wrong} attempted to show that such a finite hierarchy always exists for finitely presented $2$--torsion-free groups over any elementary family of subgroups. Unfortunately their paper contains a fatal error which has been pointed out by Louder and Touikan \cite{LouderTouikan_strong}. In the same paper Louder and Touikan prove a weaker version of this result where an ascending chain condition is required to hold as well as showing that many hierarchies of JSJ-decompositions over slender edge groups are finite. (Recall that a group is \emph{slender} if all its subgroups are finitely generated.) It is this final result and argument that this paper attempts to give a clear and detailed account of. (See \thref{res:main} for the precise statement.)

\subsection*{Acknowledgement} 

While writing this paper, the author was supported by an EPSRC-funded studentship.

\section{Preliminaries}

We will begin by recalling the different actions a subgroup $K \leqslant G$ can have on a tree $T$. A group element $g \in G$ either acts trivially on some subtree of $T$ or has an axis consisting of points which are moved a minimal amount by $g$ \cite{Serre_trees}. In the former case we call $g$ \emph{elliptic} and in the latter case we say $g$ is \emph{hyperbolic}. This leads to the following classification of actions, which is similar to the one given in \cite{BestvinaFeighn_small} except with a distinction between linear and parabolic actions.
\begin{itemize}
\item If every element of $K$ is elliptic then there is some point in $T$ which is fixed by all of $K$ \cite{Serre_trees}. We call such an action \emph{elliptic}.
\item Suppose every hyperbolic element of $K$ has a common axis. We call this action \emph{linear} if the ends of this axis are fixed and \emph{dihedral} if they are not. A linear group can be written in the form $K \cong E*_{E}$ where $E$ is a subgroup of an edge group of $T$ and the inclusion maps are isomorphic. Meanwhile a dihedral group can be written in the form $K \cong A *_{E} B$ where $E$ is a subgroup of an edge group of $T$ and both $A$ and $B$ contain $E$ as an index $2$ subgroup. In either case if the edges of $T$ have slender stabiliser then $K$ is also slender.
\item Suppose that the axes of any two hyperbolic elements of $K$ have infinite intersection, but that no line is fixed by $K$. We call such a $K$ \emph{parabolic}. Such a group fixes a single point of $\partial T$ and is a strictly ascending HNN-extension $K \cong E*_{E}$ where $E$ is a subgroup of an edge group of $T$. Observe that $K$ has an infinitely generated subgroup which is generated by $\left\lbrace a^{t^n} \mid n \in \mathbb{Z} \right\rbrace$ where $t$ is the stable letter of the HNN-extension and $a \in E$ is not contained in the non surjective end of the HNN-extension. So $K$ is not slender.
\item Suppose that $K$ contains two hyperbolic elements whose axes have compact (possibly empty) intersection. Then $K$ contains $F_2$ as a subgroup (by the ping pong lemma) and so cannot be slender. We call $K$ \emph{hyperbolic}.
\end{itemize}

In particular a slender group can only act elliptically, linearly or dihedrally on a tree.

%(Recall that a group is \emph{slender} if all its subgroups are finitely generated.)

We will use the definition of a JSJ-decomposition given by Guirardel and Levitt. \cite{GuirardelLevitt_JSJ} First recall the definition of a minimal and reduced tree.

\begin{defn}
The action of a group $G$ on a tree $T$ is \emph{minimal} if there are no $G$--invariant proper subtrees. Such an action is said to be \emph{reduced} if either $T/G$ is a circle consisting of a single vertex and edge or the label of every vertex of valence $2$ in $T/G$ properly contains its edge groups.
\end{defn}

\begin{defn}
Given a group $G$ and sets of subgroups $\mathcal{A}$ and $\mathcal{B}$ we let $\mathrm{S}_{\mathcal{A},\mathcal{B}}$ be the set of reduced trees which $G$ acts on with edge groups in $\mathcal{A}$ and where each group in $\mathcal{B}$ acts elliptically. We will assume that $\mathcal{A}$ is closed under conjugation and taking subgroups. If  If $\mathcal{B}$ is empty then we shorten $\mathrm{S}_{\mathcal{A},\mathcal{B}}$ to $\mathrm{S}_{\mathcal{A}}$.   

A tree $T \in \mathrm{S}_{\mathcal{A},\mathcal{B}}$ is \emph{universally elliptic} if every edge group of $T$ is elliptic in any given tree in $\mathrm{S}_{\mathcal{A},\mathcal{B}}$. 

A $G$-tree $T_1$ \emph{dominates} another $T_2$ if there is a $G$-map $T_1 \rightarrow T_2$. Equivalently every vertex group of $T_1$ is elliptic in $T_2$.  

A tree $T \in \mathrm{S}_{\mathcal{A},\mathcal{B}}$ is a \emph{JSJ-tree} (over the class  $\mathcal{A}$ relative to $\mathcal{B}$) if it's universally elliptic and dominates all other universally elliptic trees in $\mathrm{S}_{\mathcal{A},\mathcal{B}}$. The graph of groups corresponding to a JSJ-tree is called a \emph{JSJ-splitting} or a \emph{JSJ-decomposition} (over the class  $\mathcal{A}$ relative to $\mathcal{B}$). 
\end{defn}

Roughly speaking, we restrict our attention to splittings which do not ``exclude'' any other, then choose a maximal one amongst these. A priori it need not be the case that a JSJ-splitting exists and in complete generality they do not. However in many important cases they do in fact exist. In particular the following is true.

\begin{lemma}\cite[Theorems 2.16 \& 2.20]{GuirardelLevitt_JSJ}\thlabel{res:JSJ_exist}
Let $G$ be finitely presented (relative to some finite set of subgroups $\mathcal{B}$). Then for any class $\mathcal{A}$ there exists some JSJ-splitting for $G$ over $\mathcal{A}$ (relative to $\mathcal{B}$) with finitely generated edge groups. If the splitting is non-relative then the vertex groups are also finitely generated. 
\end{lemma}

Uniqueness of JSJ-trees does not hold in general, although one can often find a canonical choice for one-ended groups. However all of the JSJ-trees in a class live in a common deformation space. (See \cite{GuirardelLevitt_JSJ}.) From this it follows that the stabilisers of the vertices (which aren't in $\mathcal{A}$) do not depend on the choice of JSJ-tree.  Further the vertices of a JSJ-tree can be split into two classes.

\begin{defn}
A vertex of a JSJ-tree is called \emph{rigid} if it is elliptic in every tree in $\mathrm{S}_{\mathcal{A},\mathcal{B}}$. Otherwise it is called \emph{flexible}.
\end{defn}

The flexible vertices should be thought of as analogous to the Seifert-fibred components of a JSJ-decomposition of a $3$--manifold. The following result demonstrates this.

% Where on Earth did this come from?!
%
%\begin{lemma}
%A rigid vertex which has trivial JSJ-decomposition cannot split in a non trivial way. ????? (reference?????)
%\end{lemma}
%

\begin{lemma}\cite[Theorem 6.2]{GuirardelLevitt_JSJ}\thlabel{res:flexiblegroups}
Let $v$ be a flexible vertex of a slender JSJ-tree (relative to a finite set of finitely generated subgroups $\mathcal{B}$.) Then the stabiliser $G_v$ of $v$ is either slender or QH (quadratically hanging) with slender fibre. In other words, if $G_v$ is not slender there is a short exact sequence
\begin{equation*}
0 \:\rightarrow\: A \:\rightarrow\: G_v \:\rightarrow\: \pi_1(\Sigma) \:\rightarrow\: 0
\end{equation*}
where $A$ is slender and $\Sigma$ is a $2$--dimensional orbifold with a non-trivial boundary.
\end{lemma}

Such groups have very controlled JSJ-decompositions.

\begin{lemma}
Let $G$ be a non-slender group which is QH with slender fibre. The vertices of a slender JSJ-decomposition of $G$ have slender stabiliser. 
\end{lemma}

\begin{proof}
First suppose that the fibre is trivial, so that $G = \pi_1(\Sigma)$ where $\Sigma$ is a $2$--dimensional orbifold with a non-trivial boundary. Observe that we can decompose $G$ as $G_1 * \cdots * G_k * F_r$ where $G_1, \cdots , G_k$ are the finite groups associated to the orbifold points of $\Sigma$ and $F_r$ is the free group of rank $r$. Hence $G$ has Grushko decompositions over trivial edge groups with finite vertex stabilisers. Such decompositions are exactly the slender JSJ-splittings of $G$. 

Now suppose the fibre $A$ is non-trivial. Let $G$ act on a (reduced) tree $T$ with slender edge stabilisers. If $A$ acts trivially on such a $T$ then we get an action of $\frac{G}{A} \cong \pi_1(\Sigma)$ on $T$. So if $A$ acts trivially on $T$ we can lift from $\pi_1(\Sigma)$ to see that the slender JSJ-trees of $G$ have slender vertex groups, containing $A$ with finite index. Now fix $T$. If $A$ acts elliptically on $T$ then since $A$ is normal in $G$ we have some vertex $v \in T$ such that $A \leqslant G_{gv}$ for any $g \in G$. Since $T$ is reduced for any vertex $u \in T$ there is $g \in G$ such that $u \in [v,gv]$ and so $A \leqslant G_u$. Hence $A$ acts trivially on $T$. If $A$ doesn't act elliptically on $T$ it must fix a line as it's slender. As $A$ is normal in $G$ we see that $T$ must be exactly this line. However this implies that $G$ is slender which contradicts our assumption on $G$. $\square$
\end{proof}

We will define a hierarchy using Bass-Serre trees instead of the standard method using graphs of groups. While there are advantages to both approaches the present author believes this to be better one for this argument, particularly for Section~\ref{sec:find_trees}.

\begin{defn}
A \emph{hierarchy} $\mathcal{H}$ for a group $G$ (over a class $\mathcal{A}$) is a rooted tree where for each vertex $v$ of $\mathcal{H}$ we assign a subgroup $G_v \leqslant G$ with a minimal action on a tree $T_v$ (with edge groups in $\mathcal{A}$ and) where the following conditions are satisfied.
\begin{itemize}
\item The initial vertex of $\mathcal{H}$ is assigned $G$ as its group.
\item If $T_v$ is a point for some $v \in \mathcal{H}$ then $v$ is a terminal vertex of $\mathcal{H}$.  
\item Otherwise there is a natural one-to-one correspondence between the immediate descendants of a vertex $v \in \mathcal{H}$ and the vertices of the tree $T_v$. Natural in this context means that the associated groups do not change under the correspondence. Henceforth we will often abuse notation and treat the vertices related under this correspondence as the same object.
\item $\mathcal{H}$ is `conjugation invariant'. More precisely if $w$ and $gw$ are vertices of some $T_v$ (with $g \in G_v$), then the corresponding sub-hierarchies starting at $w$ and $gw$ are identical except that all the groups and labels for the vertices are conjugated by $g$. Equivalently $G$ acts naturally on $\mathcal{H}$ and the stabiliser of any vertex is its associated group.
\end{itemize}
\end{defn}

It is clear that one reobtains the traditional definition of a hierarchy by quotienting everything by $G$.

% Maybe the fact that $G$ acts on $\mathcal{H}$ needs some more clarification?
%
%
%
%$G$ acts on $\mathcal{H}$. Suppose we have a vertex $v \in \mathcal{H}$ and $g \in G$. Let $u$ be the most recent ancestor of $v$ with $g \in G_u$. If $u = v$ then we just have $gv = v$. Otherwise let $w \in T_u$ be the unique vertex which has $v$ as its descendent. (Note that we could have $w=v$.) 
%
%
%\begin{prop}
%Suppose that every vertex 
%\end{prop}
%
%\begin{proof}
%Suppose that $G_{gw} = G_w$ for some $g \in G_v$ and vertex $w \in T_v$. If 
%$\square$\end{proof}
%

\begin{defn}
The \emph{depth} of a vertex $v$ in a hierarchy $\mathcal{H}$, denoted by $\depth(v)$, is its distance from the initial vertex. The \emph{depth} of a hierarchy $\mathcal{H}$, denoted by $\depth(\mathcal{H})$, is the supremum of the depths of its vertices. 

The \emph{$n$\textsuperscript{th}--level} of $\mathcal{H}$, denoted by $\mathcal{H}^n$, is the set of all the vertices of $\mathcal{H}$ with depth $n$. 

We say that $\mathcal{H}$ is \emph{finite} if $\depth(\mathcal{H})$ is finite and the graph of groups associated to each action is finite. Equivalently $\mathcal{H}$ has finitely many $G$--orbits of vertices.
\end{defn}

\begin{defn}
A \emph{JSJ-hierarchy} (over $\mathcal{A}$ relative to $\mathcal{B}$) is a hierarchy where the action associated to every vertex is on a JSJ-tree (over $\mathcal{A}$ relative to $\mathcal{B}$.) If the group associated to a vertex is in $\mathcal{A}$ then we insist that it is terminal.
\end{defn}

Unless mentioned otherwise all JSJ-hierarchies will be non-relative and over the class of slender subgroups. Also note that the JSJ-trees \emph{don't} need be canonical (such the Bowditch JSJ-tree \cite{Bowditch_canon_JSJ}) and can instead be any maximal splitting.  

\begin{remark}
Recall that the vertex groups of a splitting of a finitely presented group over finitely presented edge groups are themselves finitely presented. In particular JSJ-hierarchies over slender edges for finitely presented groups always exist by \thref{res:JSJ_exist}.
\end{remark}

%
%In general the vertex groups of a splitting need not be finitely presented even if the original group was. However for hyperbolic groups we have the following.
%
%
%\begin{prop}\cite[Proposition 1.2]{Bowditch_canon_JSJ}
%If $G$ is a hyperbolic group, then any graph of groups decomposition of $G$ with quasiconvex edge groups also has quasiconvex vertex groups.
%\end{prop}
%
%
%\begin{theorem}[Tits alternative]\cite[Theorem 37]{EditedGromov}\thlabel{res:titsA}
%Every subgroup of a hyperbolic group $G$ is either virtually cyclic or contains $F_2$ as a subgroup.
%\end{theorem}
%
%
%In particular every slender subgroup of a hyperbolic group is virtually cyclic and hence quasi-isometrically embeds into it. Thus since every virtually cyclic subgroup of a hyperbolic group embeds quasiconvexly we see that every vertex group of a slender splitting for a hyperbolic group is also hyperbolic. So JSJ-hierarchies always exist for hyperbolic groups. 

Subgroups which are elliptic on each level of the hierarchy will play an important role.

\begin{defn}
A subgroup of $H \leqslant G$ is $\mathcal{H}$--elliptic if it's either contained in a terminal vertex of $\mathcal{H}$ or there is an infinite sequence of vertices $\left\lbrace v_n \right\rbrace$ with $H \leqslant G_{v_n}$ for all $n$ and where each $v_{n+1}$ is an immediate descendent of $v_{n}$.
\end{defn}

Suppose that $H \leqslant G$ is non-slender. Suppose that $H \leqslant G_v$ for some $v \in \mathcal{H}$. Since $T_v$ is a slender tree $H$ can be contained in at most one vertex group of $T_v$. Thus by iterating we see that $H$ is contained in at most one vertex at each level of $\mathcal{H}$. So if $H$ is also $\mathcal{H}$-elliptic then it is contained in exactly one vertex at each level of $\mathcal{H}$ and so if $H \leqslant G_v$ then it acts elliptically on $T_v$.

\begin{lemma}\thlabel{res:JSJbound}
Suppose that $\mathcal{H}$ is a JSJ-hierarchy over slender edge groups (relative to $\mathcal{B}$) for a group $G$ and that $\mathcal{K}$ is a finite slender hierarchy also for $G$ which has terminal vertex groups which are either slender or $\mathcal{H}$-elliptic. Then $\mathcal{H}$ is also finite and moreover 
\begin{equation*}
\depth{\mathcal{H}} \leq \depth{\mathcal{K}} + 1
\end{equation*}
\end{lemma}

We'll save the proof of this for Section~\ref{sec:pass_down} as it fits in naturally with what we are doing there.

The bulk of this paper will be spent trying to massage hierarchies until they satisfy the conditions of \thref{res:goal}. As a way of measuring how far away from doing this we introduce the following notions.

\begin{defn}
A \emph{$\mathcal{H}$--complex} $X$ is a $2$--dimensional connected simplical complex with $H^1(X,\mathbb{Z}_2) = 0$ and which some $K \leqslant G$ acts on  with cell stabilisers which are either slender or $\mathcal{H}$--elliptic. 

The \emph{covolume} of a $\mathcal{H}$--complex $X$ is the number of orbits of triangles under the action. Denote this quantity as $\covol(X)$. 

A \emph{$\mathcal{H}$-structure} $\mathcal{K}$ is a finite slender hierarchy for a group $K \leqslant G$ together with a $\mathcal{H}$--complex $X_w$ acted upon by $G_w$ for each terminal vertex $w \in \mathcal{K}$. 

The \emph{covolume} of a $\mathcal{H}$--structure $\mathcal{K}$ is the sum of the covolumes of the complexes associated to its terminal vertices (modulo equivalence.) Denote this quantity as $\covol(\mathcal{K})$.
\end{defn}

In many important cases (such as for hyperbolic groups) we will be able to take all of the cell stabilisers of our $\mathcal{H}$--complexes to be slender. This will streamline a few parts of the argument. Adding $\mathcal{H}$--elliptic cell stabilisers is necessary if we wish to consider certain other applications such as relatively hyperbolic groups.

\section{Main results}

The focus of this paper shall be proving the machinery necessary for showing the following fact to be true.

\begin{theorem}\cite[Corollary 2.7]{LouderTouikan_strong}\thlabel{res:goal}
Let $G$ be a hyperbolic group which is virtually $2$--torsion-free. Then any JSJ-hierarchy for $G$ is finite.
\end{theorem}

Note in particular that this implies that residually finite hyperbolic groups are strongly accessible. We will do this by proving the main result of Louder and Touikan \cite{LouderTouikan_strong} in full generality, keeping the above goal in mind so as to keep us from getting too lost in the details. Before we can state this result we need a few technical definitions.

\begin{defn}
Let $G$ be a finitely generated group and let $\mathcal{H}$ be a slender hierarchy of $G$. We say that $G$ is \emph{$\mathcal{H}$--almost finitely presented} if it has a $\mathcal{H}$--structure with finite covolume.
\end{defn}

%This is actually a slightly more general definition than the one given by Louder and Touikan as they insist that $G$ must act on a finite covolume $\mathcal{H}$--complex directly without the possibility of splitting over a hierarchy first. However their argument also works for this expanded definition with zero extra complications.

Note that $G$ being (almost) finitely presented means that it acts freely and cocompactly on some $2$--dimensional simplicial complex $X$ with $H^1(X,\mathbb{Z}_2) = 0$. Thus this notion genuinely extends the notion of being (almost) finitely presented.

The following restriction is rather technical and is essentially defined to be the exact condition which causes a step deep in the argument to work.

\begin{defn}
$\mathcal{H}$ satisfies the \emph{ascending chain condition} (henceforth abbreviated to ACC) if every ascending chain of subgroups 
\begin{equation*}
S_1 \leqslant S_2 \leqslant S_3 \leqslant \cdots ,
\end{equation*}
where the following conditions are satisfied, eventually stabilises. (i.e. there exists $N$ such that $S_N = S_i$ for all $i \geq N$.)
\begin{itemize}
\item $S_i \leqslant G_{v_i}$ for some vertex $v_i$ of $\mathcal{H}$. Moreover $S_i$ is a subgroup of a non $\mathcal{H}$-elliptic edge group in the action on $T_{v_i}$.
\item Each $v_{i+1}$ is a descendent (although not necessarily an immediate one) of $v_i$.
\item $S_i$ is $\mathcal{H}$-elliptic.
\end{itemize} 
\end{defn}

This abstract condition is satisfied if an ACC on the slender subgroups of $G$ holds. i.e. it is enough to say that every ascending chain of slender groups of $G$ must stabilise. This is a more restrictive condition, however is much easier to internalise and is satisfied for hyperbolic groups. Indeed, given an ascending chain of slender subgroups $\left\lbrace S_i \right\rbrace$, the Tits alternative \cite[Theorem 37]{EditedGromov} implies that $S_{\infty} = \bigcup_i S_i$ is either finite or virtually $\mathbb{Z}$. Thus every infinite $S_i$ has finite index in $S_{\infty}$ and so the ACC follows for free in this case.

For the argument to work we can only consider elliptic and linear actions. Hyperbolic and parabolic actions are excluded by the fact we are working with slender groups, however we still need to prohibit dihedral actions. The following definition allows us classify when we can do this.

\begin{defn}
A slender hierarchy $\mathcal{H}$ is said to be \emph{linear} if whenever $E$ is an edge group of some $T_v$ with $v \in \mathcal{H}$ we have $E \cap G_w$ acting either elliptically or linearly on $T_w$ for any $w \in \mathcal{H}$. 

A group $G$ is said to admit a \emph{$D_{\infty}$--quotient} (over the class $\mathcal{A}$) if there exists a subgroup $H \leqslant G$ which surjects onto $D_{\infty}$ (with a kernel in $\mathcal{A}$.) Observe that if a group admits no $D_{\infty}$--quotients over slender groups then all its slender hierarchies are linear.
\end{defn}

Suppose we have subgroup $D \cong A *_C B$ of a hyperbolic group $G$ which surjects $D_{\infty}$ with slender kernel $C$. Observe that $D$ is slender, so the Tits alternative \cite[Theorem 37]{EditedGromov} now implies that $D$ is virtually cyclic. Thus we see that $A, B$ and $C$ are all finite and in particular both $A$ and $B$ must contain group elements of order $2$. Thus a $2$--torsion-free hyperbolic group doesn't admit any $D_{\infty}$--quotients over its slender subgroups.

We are finally ready to state the main result in full. Note that $\mathcal{H}$ doesn't need to be a JSJ-hierarchy for the following to hold.

\begin{theorem}\cite[Theorem 2.5]{LouderTouikan_strong}\thlabel{res:main}
Let $G$ be a group and let $\mathcal{H}$ be a linear slender hierarchy for $G$. Suppose that $G$ is finitely presented (relative to $\mathcal{B}$) and that $\mathcal{H}$ satisfies the ACC. Then there exist $N$ and $C$ such that for every vertex $v$ in $\mathcal{H}$ with depth at least $N$ there exists a finite hierarchy $\mathcal{K}_v$ for $G_v$ with $\depth(\mathcal{K}_v) \leq C$. 
\end{theorem}

Once we prove this our goal follows swiftly.

\begin{proof}[Proof of \thref{res:goal}]
First consider the case where $G$ is $2$--torsion-free. Our prior discussions show that in this case the conditions for \thref{res:main} are satisfied and so its conclusions follow. Now \thref{res:JSJbound} tells us that $\mathcal{H}$ is finite with $\depth(\mathcal{H}) \leq N + C + 1$. 

Now suppose $G$ is virtually $2$--torsion-free. Claim that if $G_0$ is a finite index subgroup of $G$ then non-slender vertex groups of a JSJ-tree of $G$ are finite index supergroups of the non-slender vertex groups of a JSJ-tree of $G_0$. From this we see that a JSJ-hierarchy for $G$ has a non-slender group at level $n$ if and only if the same holds for a JSJ-hierarchy of $G_0$; at which point the $2$--torsion-free case implies that we are done. It remains to prove the claim. Guirardel and Levitt \cite[Proposition 4.16]{GuirardelLevitt_JSJ} tell us that we can obtain a JSJ-splitting for $G$ by taking a maximal splitting over finite edge groups and replacing each vertex with a JSJ-splitting for each of the (one-ended) vertex groups. In particular we can take the JSJ-splittings for the one-ended groups to be the canonical Bowditch JSJ-splitting. \cite{Bowditch_canon_JSJ} Now since $G_0 \leqslant_{\text{f.i.}} G$ they act freely and cocompactly on a common geodesic metric space; meaning that if $T$ is a maximal tree over finite edge groups for $G$ then $T_{G_0}$ is the same for $G_0$. (Possibly after reducing away some slender vertices.) Likewise the Bowditch JSJ depends only on the topology of boundary of the underlying group, hence is invariant under quasi-isometries. Thus we can build a JSJ-tree for $G$ whose restriction to $G_0$ (after reducing) is also a JSJ-tree. The claim now follows. $\square$
\end{proof}

The same argument as the above proof of \thref{res:goal} can also be used to prove the following more general statement.

\begin{theorem}\cite[Corollary 2.6]{LouderTouikan_strong}\thlabel{res:extended_goal}
Let $G$ be a group and let $\mathcal{H}$ be a slender JSJ-hierarchy (relative to a class $\mathcal{B}$) for $G$. Suppose that $G$ is $\mathcal{H}$--almost finitely presented and doesn't admit any $D_{\infty}$--quotients. Suppose also that $\mathcal{H}$ satisfies the ACC. Then $\mathcal{H}$ is finite.
\end{theorem}

\begin{remark}
There is a separate notion of a subgroup $K \leqslant G$ being slender relative so a set of subgroups $\mathcal{B}$; satisfied if $K$ fixes a point or a line on any tree $T$ which $G$ acts on where each group in $\mathcal{B}$ fixes a vertex of $T$. We do not consider such subgroups here and bring them up only to clear up any confusion between them and the notion of a slender JSJ relative to $\mathcal{B}$.
\end{remark}

\begin{remark}
While it is possible to extract an explicit value for $C$ in terms of $G$ from the upcoming proof of \thref{res:main}, (assuming that $\mathcal{H}$ is a JSJ-hierarchy,) there is no immediately obvious way to do the same for $N$. Thus this argument does not give an explicit bound on the height of $\mathcal{H}$.
\end{remark}

\section{A Note on Canonicalness of JSJ-hierarchies}

JSJ-trees are not in general unique, however they share a common deformation space. Thus every vertex group of a JSJ-tree is either slender or is a vertex group in any other JSJ-tree for the same group \cite[pg.6]{GuirardelLevitt_JSJ}. Thus we see that most of the resulting hierarchy is identical regardless of our choices of JSJ-trees. In particular this shows that the depth of a JSJ-hierarchy is constant for a given group.

If one still feels the need to consider a more canonical object then we can define one as follows. For one ended hyperbolic groups there is a canonical JSJ-tree, called the Bowditch JSJ-tree, which is invariant under outer automorphisms. (See \cite{Bowditch_canon_JSJ} or \cite{GuirardelLevitt_cylinders} for details.) Moreover we can use Dunwoody accessibility \cite{Dunwoody_fp} (which are exactly JSJ-splittings over finite groups) to split multi-ended groups into one-ended ones and the resulting vertex groups to not depend on the exact choice of tree. Alternating between these we can thus get a hierarchy which is defined in an essentially canonical way; meaning that the vertices and their associated groups of the hierarchy are always the same. One can trivially modify the previous proof of \thref{res:goal} to work for such hierarchies, to see that these hierarchies are also finite given that the underlying group is virtually without $2$--torsion.

% Such a hierarchy is reminiscent of the Haken hierarchy for $3$--manifolds \cite{Haken}, where we alternate between JSJ--decompositions over tori and maximal splittings over spheres.

\section{Passing Hierarchies to Subgroups}\label{sec:pass_down}

The key to proving \thref{res:main} will be to successively pass some `bad' auxiliary hierarchies from one level of $\mathcal{H}$ to the next until eventually they become `good'. We will thus begin by detailing a method for splitting a hierarchy over an unrelated tree.

\begin{lemma}\thlabel{res:passdown_hierarchy}
Let $T$ be a tree which a group $G$ acts on with slender edge stabilisers and let $\mathcal{K}$ be a finite slender hierarchy for $G$. Then for each vertex $v \in T$ there is a finite slender hierarchy $\mathcal{K}_v$ for the vertex group $G_v$ with the following properties.
\begin{enumerate}
\item \label{pt:depth} $\depth(\mathcal{K}_v) \leq \depth(\mathcal{K})$ for any $v \in T$.
\item \label{pt:subgroup} For each vertex $w \in \mathcal{K}_v$ the group $G_w$ is a subgroup of some $G_{u}$ where $u \in \mathcal{K}$ and $\depth(w) \leq \depth(u)$.
\item \label{pt:JSJ} If $\mathcal{K}$ is non-trivial and $T$ is a JSJ-tree then $\depth(\mathcal{K}_v) < \depth(\mathcal{K})$ whenever $v$ is a rigid vertex of $T$.
\item \label{pt:equal} Suppose that the terminal vertices of $\mathcal{K}$ have associated groups which are either slender or act elliptically in $T$. Then the terminal vertices of $\mathcal{K}_v$ have associated groups which are either slender or equal to a terminal group of $\mathcal{K}$.
\end{enumerate}
\end{lemma}

Later we will take $T$ to be a tree in the hierarchy $\mathcal{H}$ and $\mathcal{K}$ to be one of the aforementioned auxiliary hierarchies. Louder and Touikan use the symmetric core of a product of trees to produce the $\mathcal{K}_v$; however this is not necessary as a more direct approach also works.

\begin{proof}
For each $v \in T$ we define $\mathcal{K}_v$ with property~\ref{pt:subgroup} listed above one level at a time. By definition the initial vertex of $\mathcal{K}_v$ has associated group $G_v$ which trivially satisfies the required property. So we just need a procedure to generate the action on a tree for a given vertex of $\mathcal{K}_v$.

Take a vertex $w \in \mathcal{K}_v$ with associated group $G_w$. We are given that $G_w \leqslant G_{u}$ where $u \in \mathcal{K}$ and $\depth(w) \leq \depth(u)$. If $u$ is a terminal vertex then we take $T_w$ to be a point and so $w$ is also terminal. Otherwise WLOG we may assume that $G_w$ is non-elliptic in $T_{u}$ by passing to vertex groups. We now just take $T_w$ to be the unique minimal subtree of $T_{u}$ which is invariant under $G_w$. By definition the edge groups of $T_{w}$ are slender and the vertex groups of $T_{w}$ are subgroups of the vertex groups of $T_G$. The latter implies property~\ref{pt:subgroup} by induction, thus completing the construction.

We now prove the remaining properties. Property~\ref{pt:depth} is just a weaker version of property~\ref{pt:subgroup} and so is immediately satisfied. Property~\ref{pt:JSJ} holds because each rigid vertex of $T$ is by definition elliptic in the top level of $\mathcal{K}$ and so in fact for every $w \in \mathcal{K}_v$ we have $u \in \mathcal{K}$ with $G_w \leqslant G_{u}$ and $\depth(w) < \depth(u)$.

For property~\ref{pt:equal} consider a terminal vertex $w \in \mathcal{K}_v$ where $G_w$ is not slender. As usual consider a vertex $u \in \mathcal{K}$ with $G_w \leqslant G_u$ and observe that we can take $u$ to be a terminal vertex of $\mathcal{K}$. Suppose that $G_u$ fixes a line of $T$. Then $G_u \cup G_v$ either is or has an index 2 subgroup which is contained in an edge stabiliser of $T$. Hence $G_w \leqslant G_u \cap G_v$ is slender which contradicts our assumption. Thus $G_u$ fixes some vertex $v' \in T$ and so we have $G_u \leqslant G_{v'}$. Since the edges of $T$ have slender stabiliser, $G_w \leqslant G_v \cap G_{v'}$ and $G_w$ is not slender we must have $v = v'$. Since $G_u$ is $\mathcal{K}$--elliptic and $G_u \leqslant G_v$ we see that $G_u$ is also $\mathcal{K}_v$--elliptic from the construction of $\mathcal{K}_v$. Recall that since neither $G_u$ nor $G_w$ are slender they are both contained in exactly one vertex for each level of $\mathcal{K}_v$. Moreover since $G_w \leqslant G_u$ they must both be contained in the same vertices. Hence $G_u \leqslant G_w$ by the definition of $G_w$ and so $G_u = G_w$. $\square$
\end{proof}

\begin{remark}\thlabel{res:single_pushforward}
Essentially the same argument as the last part shows that if $\left\lbrace w_i \right\rbrace$ is a collection of distinct terminal vertices of $\coprod_{v} \mathcal{K}_v$ with $G_{w_i} \leqslant G_u$ where $u \in \mathcal{K}$ is a terminal vertex; then at most one of the $w_i$ can be non-slender and $G_{w_i} = G_u$ for this $i$. This will be important shortly when extending the construction to $\mathcal{H}$--structures.
\end{remark}

We now have all the tools needed to prove \thref{res:JSJbound}.

\begin{proof}[Proof of \thref{res:JSJbound}] 
We prove by induction on $\depth{\mathcal{K}}$. If $\depth{\mathcal{K}} = 0$ then $\mathcal{K}$ is trivial and so $G$ is either slender or $\mathcal{H}$--elliptic; hence it cannot split in $\mathcal{H}$ and so $\mathcal{H}$ is trivial. Otherwise consider each vertex $v \in \mathcal{H}^1$ in turn. Let $\mathcal{H}_v$ be the subhierarchy of $\mathcal{H}$ with initial vertex $v$. If $G_v$ is a rigid group in the action on the tree corresponding to the initial vertex of $\mathcal{H}$ then \thref{res:passdown_hierarchy} implies that we have another hierarchy $\mathcal{K}_v$ for $G_v$ with $\depth(\mathcal{K}_v) < \depth(\mathcal{K})$ and with terminal vertices which have associated groups that are either slender or $\mathcal{H}$--elliptic. Thus $\depth(\mathcal{H}_v) \leq \depth{\mathcal{K}}$ by induction. If $G_v$ is a flexible group in the action of top level of $\mathcal{H}$ then \thref{res:flexiblegroups} implies that $G_v$ is slender by orbifold; which implies that $\depth(\mathcal{H}_v) \leq 1 \leq \depth(\mathcal{K})$. Thus in any case we have $\depth(\mathcal{H}_v) \leq \depth(\mathcal{K})$ for all $v \in \mathcal{H}^1$ and hence $\depth(\mathcal{H}) \leq \depth(\mathcal{K}) + 1$. $\square$
\end{proof}

We will measure how `bad' our auxiliary hierarchies are by introducing some actions on some complexes. The following lemma shows us that these actions pass down nicely if the terminal vertices of our initial auxiliary hierarchy are elliptic in our tree.

\begin{lemma}\thlabel{res:passdown_complex}
Let $\mathcal{K}$ be a $\mathcal{H}$--structure for $K \leqslant G$. Suppose $K$ acts on a tree $T$ with slender edge stabilisers. Suppose that the terminal vertices of $\mathcal{K}$ are either slender or elliptic in $T$. Then for each vertex $v \in T$ we get that $\mathcal{K}_{v}$ (as defined in \thref{res:passdown_hierarchy}) naturally inherits a $\mathcal{H}$--structure from $\mathcal{K}$ and moreover we have
\begin{equation*}
\sum\limits_{i} \covol(\mathcal{K}_{v_i}) = \covol(\mathcal{K})
\end{equation*}
where $\left\lbrace v_i \right\rbrace$ is a set of representatives for the orbits of vertices in $T$.
\end{lemma}

\begin{proof}
Recall from \thref{res:passdown_hierarchy} that for every terminal vertex $w \in \mathcal{K}_v$ that $K_w$ is either slender or equal to $K_{u}$ for some terminal vertex $u \in \mathcal{K}$. Thus $\mathcal{K}_v$ naturally inherits a $\mathcal{H}$--structure by letting $K_w$ act trivially on a point if it's slender or on the same complex as $K_{u}$ otherwise. 

It remains to prove the equality of covolumes. Let $u$ be a terminal vertex of $\mathcal{K}$ and let $\left\lbrace w_j \right\rbrace$ be a set of representatives for the terminal vertices of $\coprod_i \mathcal{K}_{v_i}$ which have $K_{w_j}$ conjugating into $K_u$. Observe that for each $w_j \in \mathcal{K}_{v_{i(j)}}$ we have a corresponding $w'_j \in \mathcal{K}_{g_jv_{j(i)}}$ with $(K_{w_j})^{g_j} = K_{w'_j} \leqslant K_u$. These $w'_j$ are distinct as the $w_j$ are in distinct orbits. Hence by \thref{res:single_pushforward} at most one of the $K_{w_j}$ can be non slender and $K_{w'_j} = K_u$ for this $j$. Such a $j$ must exist as $K_u$ is contained in some vertex group $G_{v'} \leq G$ and moreover is $\mathcal{K}_{v'}$--elliptic. Thus there is a natural $G$--equivariant one to one correspondence between the non-slender terminal vertices of $\mathcal{K}$ and $\coprod_{v \in T}\mathcal{K}_v$ which implies the result. $\square$
\end{proof}

Of course in general it won't be the case that the terminal vertices of $\mathcal{K}$ will be elliptic in $T$. So our next goal shall now be to add additional layers to $\mathcal{K}$ so this becomes true while not increasing the covolume. We thus require some sort of resolution in order to nicely split up our complexes. If every cell of $X$ acted elliptically on $T$ then we could get this from an equivariant map $X \rightarrow T$. Since the cells can also act linearly on $T$ such a map needn't exist and so we need to be more careful. We will modify $X$ in order to isolate these bad points.

After making modifications to the complexes we may find that they fail to be simplical. For example if we collapse one edge of a triangle we are left with bigon. We will get around this by \emph{reducing} complexes as follows. Let $X$ be a $2$--dimensional cell complex where all the $2$--cells are either triangles or bigons. Start defining a simplical complex $X'$ by letting the its vertex set be the same as $X$. Let $[u,v]$ be an edge of $X'$ if there is an edge between $u$ and $v$ in $X$ and similarly let $[u,v,w]$ be a triangle in $X'$ if there is a triangle in $X$ with vertices $u$, $v$ and $w$. This $X'$ is the reduction of $X$. Note that if $X$ is connected with $H^1(X,\mathbb{Z}_2) = 0$ then the same holds for $X'$.

% Cannot include parabolic actions. (consider a linear action sandwiched between 2 parabolic actions, the fixed ends of the parabolic actions being each of the two ends of the linear one.)

In order to split the complex we will make use of the following Dunwoody-Delzant-Potyagailo resolution.

\begin{lemma}\thlabel{res:Dun_Del_Pot_res}\cite[Lemma 3.5]{LouderTouikan_strong}
Let $G$ be a group acting on a triangle complex $X$ and a tree $T$. Suppose that 
\begin{itemize}
\item the cell stabilisers of $X$ fix a point of $\hat{T} := T \: \cup \: \partial T$. (Where $\partial T$ is the \emph{Gromov boundary} of $T$.) i.e. they all act either elliptically, linearly or parabolically on $T$.
\item if $W \subset X^1$ is a connected subcomplex where the stabiliser of each edge acts either linearly or parabolically on $T$ then the stabiliser of $W$ fixes a point on $\partial T$.
\end{itemize}
Then there is a resolution $\rho: X \rightarrow \hat{T}$.
\end{lemma}

Before constructing $\rho$ we will show that the second condition is always satisfied for our purposes.

\begin{lemma}\thlabel{res:con_res_defn}
Let $G$ be a group acting on a tree $T$ and let $K \leq G$ act on a $\mathcal{H}$--complex $X$. Let $W \subset X^{1}$ be a connected subcomplex such that the stabiliser of every cell in $W$ acts linearly or dihedrally on $T$. Then $\stab(W)$ also acts linearly or dihedrally on $T$. Moreover if $G$ doesn't admit any $D_{\infty}$-quotients then the action of $\stab(W)$ on $T$ is linear.
\end{lemma}

\begin{proof}
Suppose that $E \leq V$ are subgroups of $G$ which both act linearly or dihedrally on $T$. Then $V$ must fix the same line as $E$ as otherwise $V$ would contain hyperbolic group elements with different axes, contradicting the fact that $V$ fixes a line of $T$. It follows that every cell in $W$ fixes a common line in $T$ which implies the result.  $\square$
\end{proof}

\begin{proof}[Proof of \thref{res:Dun_Del_Pot_res}]
First for each maximal subcomplex $W$ as in the second condition of the statement we (equivariantily) choose a point in $\partial T$ which it fixes.

Next we need to (equivariantily) map each vertex $v$ of $X$ to either a vertex of $T$ or a point on $\partial T$. If a vertex is contained in a subcomplex $W$ as above we define $\rho(v)$ to be equal to the point on $\partial T$ corresponding to $W$. Otherwise we just define $\rho(v)$ to be any vertex of $T$ which $\stab(v)$ fixes.

Now let $e = [u,v]$ be an edge of $X$. If $\rho(u) = \rho(v)$ then we can just take $\rho$ to be constant on $e$. Otherwise we want $\rho(e)$ to be the reduced edge path from $\rho(u)$ to $\rho(v)$; however we need to be careful with the parametrisation if $\stab(e) \neq \stab^{+}(e)$. If $\rho(u)$ and $\rho(v)$ are in $T$ we can just take the parametrisation to be linear as the midpoint of $[\rho(u),\rho(v)]$ is fixed by $\stab(e)$. If $\rho(u)$ is in $T$ but $\rho(v)$ isn't then $u$ and $v$ are in different orbits and so $\stab(e) = \stab^{+}(e)$. Finally suppose $\rho(u)$ and $\rho(v)$ are in $\partial T$. By assumption $\stab(e)$ fixes a vertex $v \in T$. Let $y \in [\rho(u),\rho(v)]$ be the closest vertex to $x$. Since $\stab(e)$ fixes $x$ and preserves $[\rho(u),\rho(v)]$ it must also fix $y$. Now map the midpoint of $e$ to $y$ and map the rest in any way which is symmetric through $y$. 

Extending the map affinely over triangles works for similar reasons. $\square$
\end{proof}

There are two different cases we will consider, depending on if an edge gets mapped to $\partial T$ or not. If $\rho^{-1}(\partial T)$ doesn't contain any edges then we say that this is a resolution of type I or a \emph{splitting resolution}. Otherwise $\rho^{-1}(\partial T)$ contains an edge and we call this a type II or a \emph{contracting resolution}. The splitting resolutions will allow us to modify the hierarchy so that its terminal vertices are elliptic in $T$. The contracting resolutions are an issue but we will modify them so that they become splitting resolutions.

First we will describe what to do in the case of a splitting resolution. Let $\rho: X \rightarrow \hat{T}$ be as above. Let $\Lambda \subset X$ be the inverse images of the midpoints of the edges in $T$ and observe that this is a collection of tracks (in the sense of Dunwoody \cite{Dunwoody_fp}). Let $\Lambda^* \subset \Lambda$ be the tracks which partition $X$ into two infinite parts and let $X^* := X \backslash \rho^{-1}(\partial T)$. Observe that $X^* / \Lambda^*$ (obtained by collapsing each track in $\Lambda^*$ to a point) is a $2$-dimensional cell complex where all the $2$--cells are either bigons or triangles. Finally $X_T$ is defined to be the reduction of $X^* / \Lambda^*$. Observe that the image of each triangle of $X$ in $X_T$ contains at most one triangle and so $\covol(X_T) \leq \covol(X)$.

%% Accompanies image file 'TriangleWithTracks.pdf' (pdf, eps, ps)
\begingroup%
  \makeatletter%
  \providecommand\color[2][]{%
    \errmessage{(Inkscape) Color is used for the text in Inkscape, but the package 'color.sty' is not loaded}%
    \renewcommand\color[2][]{}%
  }%
  \providecommand\transparent[1]{%
    \errmessage{(Inkscape) Transparency is used (non-zero) for the text in Inkscape, but the package 'transparent.sty' is not loaded}%
    \renewcommand\transparent[1]{}%
  }%
  \providecommand\rotatebox[2]{#2}%
  \newcommand*\fsize{\dimexpr\f@size pt\relax}%
  \newcommand*\lineheight[1]{\fontsize{\fsize}{#1\fsize}\selectfont}%
  \ifx\svgwidth\undefined%
    \setlength{\unitlength}{413.85826772bp}%
    \ifx\svgscale\undefined%
      \relax%
    \else%
      \setlength{\unitlength}{\unitlength * \real{\svgscale}}%
    \fi%
  \else%
    \setlength{\unitlength}{\svgwidth}%
  \fi%
  \global\let\svgwidth\undefined%
  \global\let\svgscale\undefined%
  \makeatother%
  \begin{picture}(1,0.26027397)%
    \lineheight{1}%
    \setlength\tabcolsep{0pt}%
    \put(0.70431628,-0.41812724){\color[rgb]{0,0,0}\makebox(0,0)[lt]{\begin{minipage}{0.12438615\unitlength}\centering \end{minipage}}}%
    \put(0.03662987,0.3576232){\color[rgb]{0,0,0}\makebox(0,0)[lt]{\begin{minipage}{0.93116981\unitlength}\centering \end{minipage}}}%
    \put(0,0){\includegraphics[width=\unitlength,page=1]{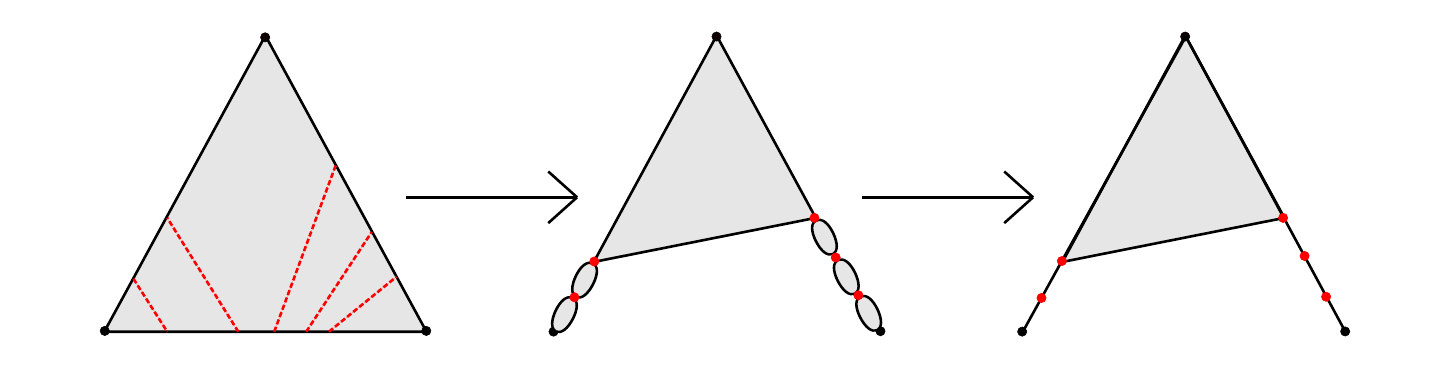}}%
    \put(0.26048117,0.13939761){\color[rgb]{1,0,0}\makebox(0,0)[t]{\lineheight{1.25}\smash{\begin{tabular}[t]{c}$\Lambda^*$\end{tabular}}}}%
  \end{picture}%
\endgroup%

\begin{remark}\thlabel{res:edge_passdown}
Unlike with triangles, the image of an edge of $X$ in $X_T$ need not be a single edge and is in general a (potentially infinite) sequence of edges. However suppose we are given a triangle in $X$ whose image in $X_T$ contains a triangle. Then the edges of this new triangle in $X_T$ will be a single edge in the image of a corresponding edge of the original triangle. 
\end{remark}

\begin{lemma}\thlabel{res:split_res_sc}
Suppose that $X$ is connected, has $H^1(X, \mathbb{Z}_2) = 0$ and $\rho: X \rightarrow T$ is a splitting resolution. Suppose also that every cutpoint of $X$ has stabiliser which fixes a point of $T$. Then $X_T$ is connected with $H^1(X_T, \mathbb{Z}_2) = 0$.
\end{lemma}

\begin{proof}
Since every cutpoint of $X$ acts elliptically on $T$ we see that $X^*$ is connected and therefore $X_T$ is as well. Since each track is connected it now suffices to show that each cycle in $H^1(X^*,\mathbb{Z}_2)$ is a boundary when mapped into $X_T$. 

Let $B$ be the second barycentric subdivision of $X$. Let $C \subset B$ be the union of simplices which are disjoint from $\rho^{-1}(\partial T)$ and $A \subset B$ be the union of the simplices (and their subsimplices) which intersect non-trivially with $\rho^{-1}(\partial T)$. Also let $L = A \cap C$. 

Consider the Mayer-Vietoris sequence for $A$ and $C$;
\begin{equation*}
\cdots \rightarrow H^{1}(X,\mathbb{Z}_2) \rightarrow H^{1}(A,\mathbb{Z}_2) \oplus H^{1}(C,\mathbb{Z}_2) \rightarrow H^{1}(L,\mathbb{Z}_2) \rightarrow \cdots
\end{equation*}
Since $B$ is the second barycentric subdivision of $X$ we see that the stars of two distinct vertices which are in $\rho^{-1}(\partial T)$ are disjoint. Hence each component of $A$ is the star of a point and so $H^{1}(A,\mathbb{Z}_2) \cong 0$. Similarly we see that $C \subset X^*$ is a deformation retract. Hence the sequence becomes 
\begin{equation*}
0 \rightarrow H^{1}(X^{*},\mathbb{Z}_2) \rightarrow H^{1}(L,\mathbb{Z}_2) \rightarrow \cdots
\end{equation*}
In particular it suffices to show that each loop of $L$ dies when we pass to $X_T$. 

%% Accompanies image file 'ResolutionSC.pdf' (pdf, eps, ps)
\begingroup%
  \makeatletter%
  \providecommand\color[2][]{%
    \errmessage{(Inkscape) Color is used for the text in Inkscape, but the package 'color.sty' is not loaded}%
    \renewcommand\color[2][]{}%
  }%
  \providecommand\transparent[1]{%
    \errmessage{(Inkscape) Transparency is used (non-zero) for the text in Inkscape, but the package 'transparent.sty' is not loaded}%
    \renewcommand\transparent[1]{}%
  }%
  \providecommand\rotatebox[2]{#2}%
  \newcommand*\fsize{\dimexpr\f@size pt\relax}%
  \newcommand*\lineheight[1]{\fontsize{\fsize}{#1\fsize}\selectfont}%
  \ifx\svgwidth\undefined%
    \setlength{\unitlength}{413.85826772bp}%
    \ifx\svgscale\undefined%
      \relax%
    \else%
      \setlength{\unitlength}{\unitlength * \real{\svgscale}}%
    \fi%
  \else%
    \setlength{\unitlength}{\svgwidth}%
  \fi%
  \global\let\svgwidth\undefined%
  \global\let\svgscale\undefined%
  \makeatother%
  \begin{picture}(1,0.26027397)%
    \lineheight{1}%
    \setlength\tabcolsep{0pt}%
    \put(0.03662987,0.3576232){\color[rgb]{0,0,0}\makebox(0,0)[lt]{\begin{minipage}{0.93116981\unitlength}\centering \end{minipage}}}%
    \put(0,0){\includegraphics[width=\unitlength,page=1]{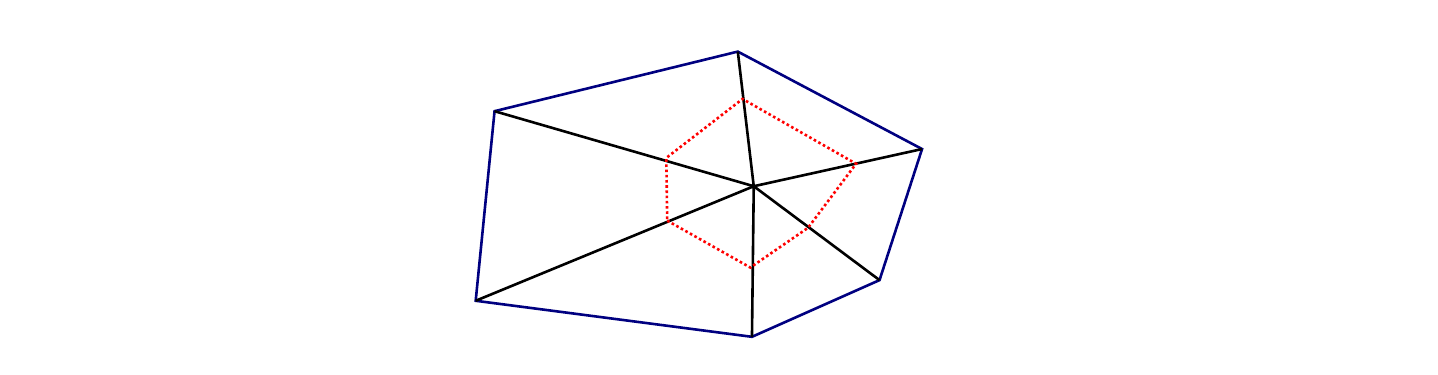}}%
    \put(0.53402323,0.14231914){\color[rgb]{0,0,0}\makebox(0,0)[t]{\lineheight{1.25}\smash{\begin{tabular}[t]{c}$v$\end{tabular}}}}%
    \put(0.55687264,0.06789534){\color[rgb]{1,0,0}\makebox(0,0)[t]{\lineheight{1.25}\smash{\begin{tabular}[t]{c}$\lambda$\end{tabular}}}}%
    \put(0.39627397,0.11424694){\color[rgb]{0,0,0}\makebox(0,0)[t]{\lineheight{1.25}\smash{\begin{tabular}[t]{c}$t_i$\end{tabular}}}}%
    \put(0.3133633,0.03655901){\color[rgb]{0,0,0}\makebox(0,0)[t]{\lineheight{1.25}\smash{\begin{tabular}[t]{c}$w_{i-1}$\end{tabular}}}}%
    \put(0.33258825,0.18867096){\color[rgb]{0,0,0}\makebox(0,0)[t]{\lineheight{1.25}\smash{\begin{tabular}[t]{c}$w_{i}$\end{tabular}}}}%
    \put(0.43022166,0.06694992){\color[rgb]{0,0,0}\makebox(0,0)[t]{\lineheight{1.25}\smash{\begin{tabular}[t]{c}$e_{i-1}$\end{tabular}}}}%
    \put(0.44393132,0.16683437){\color[rgb]{0,0,0}\makebox(0,0)[t]{\lineheight{1.25}\smash{\begin{tabular}[t]{c}$e_{i}$\end{tabular}}}}%
    \put(0.30944623,0.11685829){\color[rgb]{0,0,0.50196078}\makebox(0,0)[t]{\lineheight{1.25}\smash{\begin{tabular}[t]{c}$p_i$\end{tabular}}}}%
    \put(0.60257146,0.19940875){\color[rgb]{0,0,0.50196078}\makebox(0,0)[t]{\lineheight{1.25}\smash{\begin{tabular}[t]{c}$d$\end{tabular}}}}%
  \end{picture}%
\endgroup%

Let $d$ be a reduced edge path of $L$. There is $v \in \rho^{-1}(\partial T)$ with $d$ homotopic to an edge path $p_1 \cdots p_n$ in the link of $v$. Let $w_i$ be the vertex common to both $p_i$ and $p_{i+1}$ (where the indices are taken modulo $n$). Also let $e_i$ be the edge connecting $w_i$ to $v$ and $t_i$ be the triangle with edges $w_{i-1}$, $w_{i}$ and $v$. Since $\rho(w_i) \neq \rho(v)$ for any $i$ and $\rho(v) \in \partial T$ we can choose an edge $f \in T$ such that $f \in [\rho(w_i),\rho(v)]$ and $f \cap \rho(p_i) = \emptyset$ for all $i$. Thus there is a track $\lambda$ (which maps to the midpoint of $f$) whose intersection with $t_1 \cup \cdots \cup t_n$ is homotopic to $d$. This implies that $d$ is null-homotopic in $X_T$ which implies the result. $\square$
\end{proof}

We will now detail what to do in the case of a contracting resolution. Recall that this is the case where $\rho^{-1}(\partial T)$ contains an edge of $X$. Define $X_C$ to be the complex obtained by collapsing each component of $\rho^{-1}(\partial T) \subset X$ to a point. We summarise the properties of $X_C$ in the following lemma.

\begin{lemma}\thlabel{res:con_res_stabs}
Notation as before. Suppose that the cell stabilisers of $X$ are all either slender or elliptic in $T$. Then every vertex stabiliser of $X_C$ is either equal to a vertex stabiliser of $X$ or fixes a line of $T$ and hence is slender. Moreover $\rho$ descends to a map $X_C \rightarrow \hat{T}$ where the midpoint of each edge of $X_C$ gets sent to a point in $T$ and $\covol(X_C) \leq \covol(X)$. 
\end{lemma}

\begin{proof}
These properties are all obvious from the definition of $X_C$ together with \thref{res:con_res_defn}. $\square$
\end{proof}

This concludes our discussion on contracting resolutions. We also need a general method for splitting a complex over its cutpoints. Suppose $X$ is a $\mathcal{H}$--complex for a subgroup $K \leqslant G$.  Construct a bipartite tree $B_X$, called the \emph{cutpoint tree}, with vertices which correspond to the cutpoint-free components of $X$ and the cut vertices of $X$ and with the edges of $B_X$ defined in the obvious way by inclusion. Note that $K$ acts on $B_X$ with vertex stabilisers which are equal to the corresponding stabilisers in $X$ and edge stabilisers which are equal to the stabiliser of a connected component of a link.

The above is fine if every cell of $X$ has slender stabiliser, such as in the case of hyperbolic groups; however in general it may be the case that $B_X$ has edge groups which are $H$--elliptic but not slender. To counteract this we introduce a \emph{reduced cutpoint tree} $B'_X$ by collapsing the edges of $B_X$ which have non-slender stabiliser. Observe that $B'_X$ can naturally be thought of as a $\mathcal{H}$--structure (of depth $1$) with the properties summarised in the following lemma.

\begin{lemma}\thlabel{res:cutpoint_split}
Let $B'_X$ be the cutpoint tree defined above.
\begin{enumerate}
\item \label{pt:cutpoint_covol} $\covol(B'_X) \leq \covol(X)$
\item \label{pt:cutpoint_elliptic} If $X_T$ is the splitting resolution of some complex $X$ where the cutpoints of $X$ have non-slender stabilisers then the vertex groups of $B'_{X_T}$ are either slender or elliptic in $T$.
\end{enumerate}
\end{lemma}

\begin{proof}
For part \ref{pt:cutpoint_covol} we first observe that each triangle of $X$ sits in at most one subcomplex corresponding to a non-slender vertex of $B'_X$. Thus we just need to know if two triangles $t$ and $t' = gt$ (with $g \in K$) live in the same subcomplex $X_v \subseteq X$ corresponding to a vertex $v \in B'_{X}$ then they still lie in the same orbit when restricted to the vertex group $K_v$. This follows because if $g \in K$ sends a triangle of $X_v$ to another in $X_v$ then it must preserve $X_v$, so $g \in K_v$. 

For part \ref{pt:cutpoint_elliptic} we first observe that the stabiliser of each connected component $Y \subseteq X^* \backslash \Lambda^*$ is elliptic in $T$ by construction. The image of $\bar{Y}$ in $X_T$ is a maximal subcomplex $Y'$ which doesn't contain any cutpoints with slender stabiliser. Moreover the stabiliser of $Y$ is the same as the stabiliser of $Y'$. A subcomplex of $X_T$ corresponding to a non-slender vertex of $B'_{X_T}$ is contained in such a $Y'$. Hence the non-slender vertices of $B'_{X_T}$ are elliptic in $T$. $\square$
\end{proof}

Putting all of this together we obtain the following.

\begin{lemma}\thlabel{res:passdown_full}\cite[Lemma 7.2]{LouderTouikan_strong}
Suppose $G$ does not admit any $D_{\infty}$--quotients. Let $T$ be a tree which $K \leqslant G$ acts on with slender edge stabilisers and let $\mathcal{K}$ be a $\mathcal{H}$--structure for $K$. Then for each vertex $v \in T$ there is a $\mathcal{H}$--structure $\mathcal{K}_v$ with
\begin{equation*}
\sum_i \covol(\mathcal{K}_{v_i}) \leq \covol(\mathcal{K})
\end{equation*}
where $\left\lbrace v_i \right\rbrace$ is a set of representatives for the orbits of vertices of $T$.
\end{lemma}

\begin{proof}
In light of \thref{res:passdown_complex} it suffices to show that there is another $\mathcal{H}$--structure $\tilde{\mathcal{K}}$ with terminal vertices which are elliptic in $T$ and with $\covol(\tilde{\mathcal{K}}) \leq \covol(\mathcal{K})$.

Begin by considering the resolution of each complex associated to the terminal vertices of $\mathcal{K}$. We define $\mathcal{K}'$ to be the same as $\mathcal{K}$ except whenever a complex $X$ associated to a terminal vertex has an edge mapped into $\partial T$ we replace it with $X_C$ defined above. Recall from \thref{res:con_res_stabs} that $X_C$ is a $\mathcal{H}$--complex and $\covol(X_C) \leq \covol(X)$. Thus $\mathcal{K}'$ is a $\mathcal{H}$--structure with $\covol(\mathcal{K}') \leq \covol(\mathcal{K})$. Moreover the resolutions from each complex are now all splitting resolutions.

Before collapsing tracks it is first necessary to split the complexes over cutpoints. (Otherwise our $X_T$ may not be connected.) Following the procedure from \thref{res:cutpoint_split} for each complex of $\mathcal{K}'$ containing cutpoints we construct a new $\mathcal{H}$--structure $\mathcal{K}''$ where all the cutpoints in the complexes have $\mathcal{H}$--elliptic stabilisers. Moreover the resolutions still map each edge of the complexes to a point of $T$ and $\covol(\mathcal{K}'') \leq \covol(\mathcal{K}')$.

Finally for each complex $X$ of $\mathcal{K}''$ we consider $X_T$ as defined above. Recall from \thref{res:cutpoint_split} that each non-slender vertex of $B'_{X_T}$ acts elliptically on $T$. Thus we get a new $\mathcal{H}$--structure $\tilde{\mathcal{K}}$ by replacing each terminal vertex of $\mathcal{K}''$ with the corresponding $B'_{X_T}$. This $\tilde{\mathcal{K}}$ has all the properties we require. $\square$
\end{proof}

\begin{remark}\thlabel{res:triangle_passdown}
Note that there is a natural partial map from the set of triangles in the complexes of $\mathcal{K}$ and those in $\coprod_{v \in T} \mathcal{K}_v$. Moreover this map is $G$--equivarient and is both total and bijective if $\sum_i \covol(\mathcal{K}_{v_i}) = \covol(\mathcal{K})$ where $\left\lbrace v_i \right\rbrace$ is a set of representatives for the orbits of vertices of $T$. An understanding of this map will be crucial for Section \ref{sec:find_trees}.
\end{remark}

\section{Extracting Trees from Complexes}\label{sec:find_trees}

Now with \thref{res:passdown_full} in hand we are ready to start the proof of \thref{res:main}. Let $v_0$ be the initial vertex of $\mathcal{H}$. Start by letting $\mathcal{K}_{v_0}$ be any $\mathcal{H}$--structure for $G$ with finite covolume. (Recall that for a finitely presented group we can take $\mathcal{K}_{v_0}$ to have trivial tree structure and have $G$ act freely on a cocompct ($2$--dimensional) simply connected simplical complex.) Now we recursively define $\mathcal{K}_w$ for each vertex $w \in \mathcal{H}$. Suppose $w'$ is the immediate ancestor of $w$ and $\mathcal{K}_{w'}$ is already defined. We now define $\mathcal{K}_w$ from \thref{res:passdown_full} by setting $\mathcal{K}$ to be $\mathcal{K}_{w'}$ and $T$ to be $T_{w'}$.

Let $\mathcal{T}^n$ be the set of all the triangles in all the complexes acted on by the terminal vertices of $\mathcal{K}_w$ where $w \in \mathcal{H}^n$. Note that $G$ naturally acts on $\mathcal{T}^n$ with finitely many orbits of triangles; call this number $\covol(\mathcal{T}^n)$. Moreover the inequality of covolumes in \thref{res:passdown_full} extends to an inequality $\covol(\mathcal{T}^{n+1}) \leq \covol(\mathcal{T}^n)$ for all $n$. Thus $\covol(\mathcal{T}^n)$ must eventually reach some minimum. Pick $N_{\Delta}$ so that $\covol(\mathcal{T}^{N_{\Delta}}) = \covol(\mathcal{T}^n)$ for any $n \geq N_{\Delta}$.

Recall from \thref{res:triangle_passdown} that for $n \geq N_{\Delta}$ we can always pass a triangle to the next level of $\mathcal{H}$. More precisely our construction above actually induces a $G$--equivariant bijective map $\tau_{n,n+1} : \mathcal{T}^n \rightarrow \mathcal{T}^{n+1}$. Moreover let $\tau_{n,m} : \mathcal{T}^n \rightarrow \mathcal{T}^m$ be the composition $\tau_{m-1,m} \circ \cdots \circ \tau_{n,n+1}$.

A \emph{pair} in $\mathcal{T}^n$ is an element $(t,t') \in \mathcal{T}^n \times \mathcal{T}^n$ where $t$ and $t'$ are distinct triangles in the same complex and which share a common edge $e$. A pair is called \emph{stable} if it descends to a pair under any $\tau_{n,m}$ where $m>n$. Let $P(\mathcal{T}^n)$ be the set of stable pairs in $\mathcal{T}^n$.

We now define an equivalence relation $\sim_n$ on $\mathcal{T}^n$ to be the one generated by its stable pairs. Note that for each equivalence class of $\sim_n$ we naturally get a connected subcomplex (of some $\mathcal{H}$--complex which is associated to a terminal vertex of $\mathcal{K}_v$ for some $v \in \mathcal{H}^n$) consisting of all the triangles in the class together with all their subsimplices.

We now restrict our attention to a single complex $X_w$ associated to a terminal vertex $w \in\mathcal{K}_v$ for some $v \in \mathcal{H}^n$ with $n \geq N_{\Delta}$. We define a bipartite graph $B_w$ for each $X_w$ as follows. One set of vertices will be the set of subcomplexes associated to the equivalence classes of $\sim_n$ which are contained in $X_w$; the other will be the edges of $X_w$ which are contained in more than one of said subcomplexes. The edges of $B_w$ are defined by inclusion in the obvious way.

Observe that $G_w$ acts naturally on $B_w$. By definition the stabilisers for the subcomplexes associated to the equivalence classes of $\sim_n$ are $\mathcal{H}$--elliptic. If every edge of $X_w$ has slender stabiliser (such as in the case for hyperbolic groups) then the stabilisers of each edge of $B_w$ are slender. So if $B_w$ is a tree for all large enough $n$ then this proves \thref{res:main} by adding the $B_w$ to the bottom layers of the $\mathcal{K}_v$. (Where $N$ in the statement of \thref{res:main} is the first level where this occurs and the corresponding $C$ is the maximal depth of one of the $\mathcal{K}_v$ where $v$ has depth $N$ in $\mathcal{H}$.) If some edge of $X_w$ has a non-slender ($\mathcal{H}$--elliptic) stabiliser then we instead first have to collapse each edge of $B_w$ with non-slender stabilser to get a new graph $B'_w$. \thref{res:main} will then follow as before.

We shall now work backwards finding a series of sufficient conditions for $B_w$ to be a tree until we arrive at one which we can show is true for large $n$. First observe that this is true if we can show that, for far enough down the hierarchy, whenever $(t,t')$ is an unstable pair with common edge $e$ that  $t$ and $t'$ lie in different connected components of $X_w \backslash e$.  We now need a definition.

\begin{defn}
Let $D$ be a triangulated disk with exactly one interior vertex. A \emph{cone} $C \subseteq X$ is the image of some simplicial map $\alpha: D \rightarrow X$ which sends triangles to triangles. A cone is said to be \emph{simple} if the image of $\partial \alpha : \partial S^1 \rightarrow X$ is a simple loop; equivalently $\alpha$ is injective.
\end{defn}

\begin{lemma}\thlabel{res:cone_criterion}\cite[Lemma 8.5]{LouderTouikan_strong}
If every simple cone of $X_w$ is contained in an equivalence class, then $B_w$ is a tree.
\end{lemma}

Before proving this we require a simple proposition.

\begin{prop}\thlabel{res:make_cones_simple_again}
Let $\gamma$ be the boundary of some cone $C$. Suppose $e = [u,v]$ and $e' = [v,w]$ are consecutive edges of $\gamma$. If $u \neq w$ (so $\gamma$ is locally injective at vertex $v$) then there a simple subcone $C' \subseteq C$ containing both $e$ and $e'$.
\end{prop}

\begin{proof}
Suppose $\gamma: S^1 \rightarrow C$ is not simple. Then there are distinct $x_1,x_2 \in S^1$ which map to some common vertex $x \in C$. Let $A$ be an arc of $S^1$ which starts at $x_1$, finishes at $x_2$ and which contains $e$ and $e'$ as consecutive edges. Let $A'$ be the circle formed by taking $A$ and gluing its endpoints together. Then $\gamma |_{A'}$ is the boundary for a proper subcone of $C$ which contains $e$ and $e'$ as consecutive edges. Repeat this process until the resulting cone is simple, which must happen eventually as the area of the cone decreases at each step. $\square$
\end{proof}

\begin{proof}[Proof of \thref{res:cone_criterion}]
Let $(t,t')$ be pair in $X_w$ with common edge $e$. In order to prove the result it suffices to show that if $t$ and $t'$ are in the same connected component of $X_w \backslash e$ then $t \sim t'$. In this case $e = [u,v]$ is not a cut edge of $X_w$. Let $a$ and $b$ be the vertices of $t$ and $t'$ respectively that are not a part of $e$. Since $e$ is not a cut edge there is an edge path $\gamma$ (which we'll not assume is injective) from $a$ to $b$ which doesn't intersect $e$. Let $l$ be the loop consisting of $\gamma$ composed with $p = [a,u] \cup [u,b]$. Since $X$ is simply connected there is a simply connected simplical complex $D \subset \mathbb{R}^2$ together with a simplical map $\rho:D \rightarrow X$ with boundary $\partial\rho:\partial D \rightarrow l$. Note that $\rho$ is not required to be an embedding, even locally so. We will now assume that $\gamma$, $D$ and $\rho$ as above are chosen to lexicographically optimise the following quantities for which $D$ is homeomorphic to a disc. ($D$ is always homeomorphic to a disc if $\gamma$ is injective; but this needn't be the case in general.)
\begin{itemize}
\item Minimises the number of triangles in $D$.
\item Maximises the length of $\partial D$.
\end{itemize}
Note the length of $\partial D$ is bounded above by three times the number of triangles of $D$. Thus we have a well ordering and so an optimal choice must exist. 

For such optimal choices we get the following properties.
\begin{itemize}
\item Since $\rho$ is a homeomorphism on each simplex we see that two disjoint components of $\rho^{-1}(e)$ must be separated by an edge path $\lambda$ in $D$. Hence $\rho^{-1}(e)$ is connected in an optimal choice as otherwise we can `cut across' $\lambda$ to get a new loop which bounds strictly less area. 

%% Accompanies image file 'PreimageConnected.pdf' (pdf, eps, ps)
\begingroup%
  \makeatletter%
  \providecommand\color[2][]{%
    \errmessage{(Inkscape) Color is used for the text in Inkscape, but the package 'color.sty' is not loaded}%
    \renewcommand\color[2][]{}%
  }%
  \providecommand\transparent[1]{%
    \errmessage{(Inkscape) Transparency is used (non-zero) for the text in Inkscape, but the package 'transparent.sty' is not loaded}%
    \renewcommand\transparent[1]{}%
  }%
  \providecommand\rotatebox[2]{#2}%
  \newcommand*\fsize{\dimexpr\f@size pt\relax}%
  \newcommand*\lineheight[1]{\fontsize{\fsize}{#1\fsize}\selectfont}%
  \ifx\svgwidth\undefined%
    \setlength{\unitlength}{413.85826772bp}%
    \ifx\svgscale\undefined%
      \relax%
    \else%
      \setlength{\unitlength}{\unitlength * \real{\svgscale}}%
    \fi%
  \else%
    \setlength{\unitlength}{\svgwidth}%
  \fi%
  \global\let\svgwidth\undefined%
  \global\let\svgscale\undefined%
  \makeatother%
  \begin{picture}(1,0.26027397)%
    \lineheight{1}%
    \setlength\tabcolsep{0pt}%
    \put(0.03662987,0.3576232){\color[rgb]{0,0,0}\makebox(0,0)[lt]{\begin{minipage}{0.93116981\unitlength}\centering \end{minipage}}}%
    \put(0,0){\includegraphics[width=\unitlength,page=1]{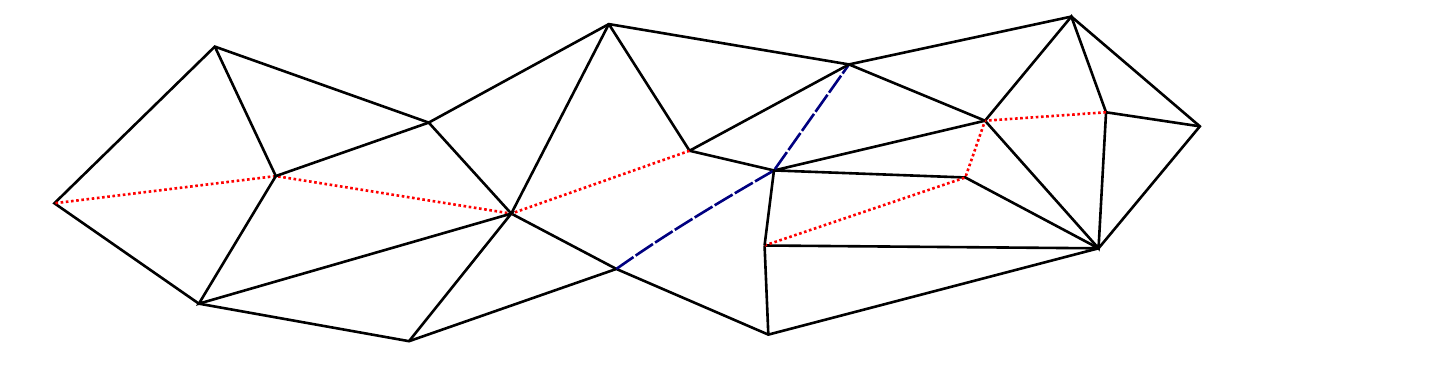}}%
    \put(0.13056803,0.14101337){\color[rgb]{1,0,0}\makebox(0,0)[t]{\lineheight{1.25}\smash{\begin{tabular}[t]{c}$\rho^{-1}(e)$\end{tabular}}}}%
    \put(0,0){\includegraphics[width=\unitlength,page=2]{PreimageConnected.pdf}}%
    \put(0.64200629,0.09982363){\color[rgb]{1,0,0}\makebox(0,0)[t]{\lineheight{1.25}\smash{\begin{tabular}[t]{c}$\rho^{-1}(e)$\end{tabular}}}}%
    \put(0.58284744,0.16932657){\color[rgb]{0,0,0.50196078}\makebox(0,0)[t]{\lineheight{1.25}\smash{\begin{tabular}[t]{c}$\lambda$\end{tabular}}}}%
  \end{picture}%
\endgroup%

\item Every edge $f \in \partial D$ must be in the link of a preimage of either $u$ or $v$. Otherwise we could remove $f$ and the unique triangle which contains $f$ to obtain a new loop which bounds strictly less area.

%% Accompanies image file 'RemoveTriangle.pdf' (pdf, eps, ps)
\begingroup%
  \makeatletter%
  \providecommand\color[2][]{%
    \errmessage{(Inkscape) Color is used for the text in Inkscape, but the package 'color.sty' is not loaded}%
    \renewcommand\color[2][]{}%
  }%
  \providecommand\transparent[1]{%
    \errmessage{(Inkscape) Transparency is used (non-zero) for the text in Inkscape, but the package 'transparent.sty' is not loaded}%
    \renewcommand\transparent[1]{}%
  }%
  \providecommand\rotatebox[2]{#2}%
  \newcommand*\fsize{\dimexpr\f@size pt\relax}%
  \newcommand*\lineheight[1]{\fontsize{\fsize}{#1\fsize}\selectfont}%
  \ifx\svgwidth\undefined%
    \setlength{\unitlength}{413.85826772bp}%
    \ifx\svgscale\undefined%
      \relax%
    \else%
      \setlength{\unitlength}{\unitlength * \real{\svgscale}}%
    \fi%
  \else%
    \setlength{\unitlength}{\svgwidth}%
  \fi%
  \global\let\svgwidth\undefined%
  \global\let\svgscale\undefined%
  \makeatother%
  \begin{picture}(1,0.26027397)%
    \lineheight{1}%
    \setlength\tabcolsep{0pt}%
    \put(0.03662987,0.3576232){\color[rgb]{0,0,0}\makebox(0,0)[lt]{\begin{minipage}{0.93116981\unitlength}\centering \end{minipage}}}%
    \put(0,0){\includegraphics[width=\unitlength,page=1]{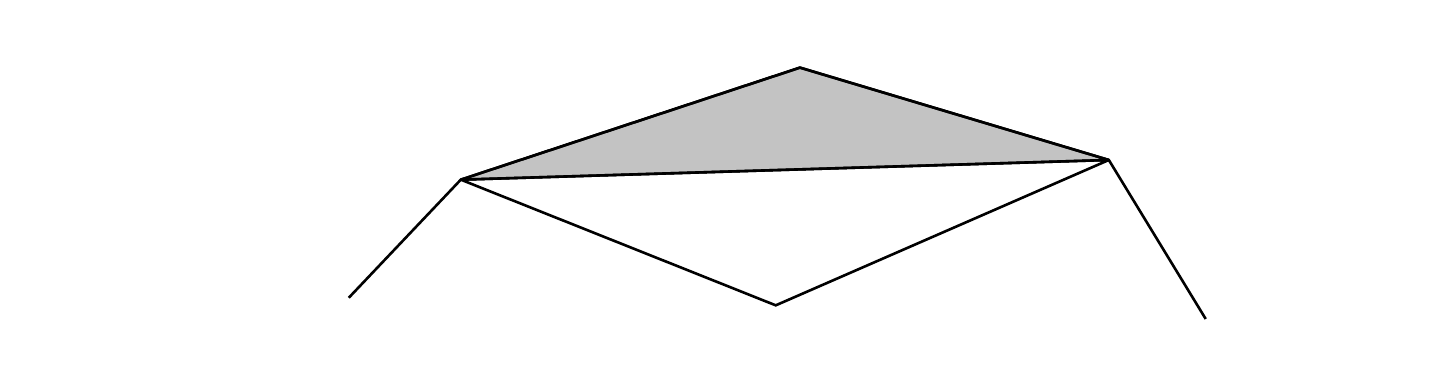}}%
    \put(0.23777941,0.18699479){\color[rgb]{0,0,0}\makebox(0,0)[t]{\lineheight{1.25}\smash{\begin{tabular}[t]{c}Delete this\\triangle\end{tabular}}}}%
    \put(0,0){\includegraphics[width=\unitlength,page=2]{RemoveTriangle.pdf}}%
    \put(0.6655229,0.19002042){\color[rgb]{0,0,0.50196078}\makebox(0,0)[t]{\lineheight{1.25}\smash{\begin{tabular}[t]{c}$f$\end{tabular}}}}%
  \end{picture}%
\endgroup%

\item The only non-boundary vertices in the link of a vertex $w'$ of $\partial D$ are preimages of $u$ and $v$. Otherwise we could make $\partial D$ longer by adding two copies of an interior edge of $D$ to $\partial D$. 

%% Accompanies image file 'CutToAddEdges.pdf' (pdf, eps, ps)
\begingroup%
  \makeatletter%
  \providecommand\color[2][]{%
    \errmessage{(Inkscape) Color is used for the text in Inkscape, but the package 'color.sty' is not loaded}%
    \renewcommand\color[2][]{}%
  }%
  \providecommand\transparent[1]{%
    \errmessage{(Inkscape) Transparency is used (non-zero) for the text in Inkscape, but the package 'transparent.sty' is not loaded}%
    \renewcommand\transparent[1]{}%
  }%
  \providecommand\rotatebox[2]{#2}%
  \newcommand*\fsize{\dimexpr\f@size pt\relax}%
  \newcommand*\lineheight[1]{\fontsize{\fsize}{#1\fsize}\selectfont}%
  \ifx\svgwidth\undefined%
    \setlength{\unitlength}{413.85826772bp}%
    \ifx\svgscale\undefined%
      \relax%
    \else%
      \setlength{\unitlength}{\unitlength * \real{\svgscale}}%
    \fi%
  \else%
    \setlength{\unitlength}{\svgwidth}%
  \fi%
  \global\let\svgwidth\undefined%
  \global\let\svgscale\undefined%
  \makeatother%
  \begin{picture}(1,0.26027397)%
    \lineheight{1}%
    \setlength\tabcolsep{0pt}%
    \put(0.03662987,0.3576232){\color[rgb]{0,0,0}\makebox(0,0)[lt]{\begin{minipage}{0.93116981\unitlength}\centering \end{minipage}}}%
    \put(0,0){\includegraphics[width=\unitlength,page=1]{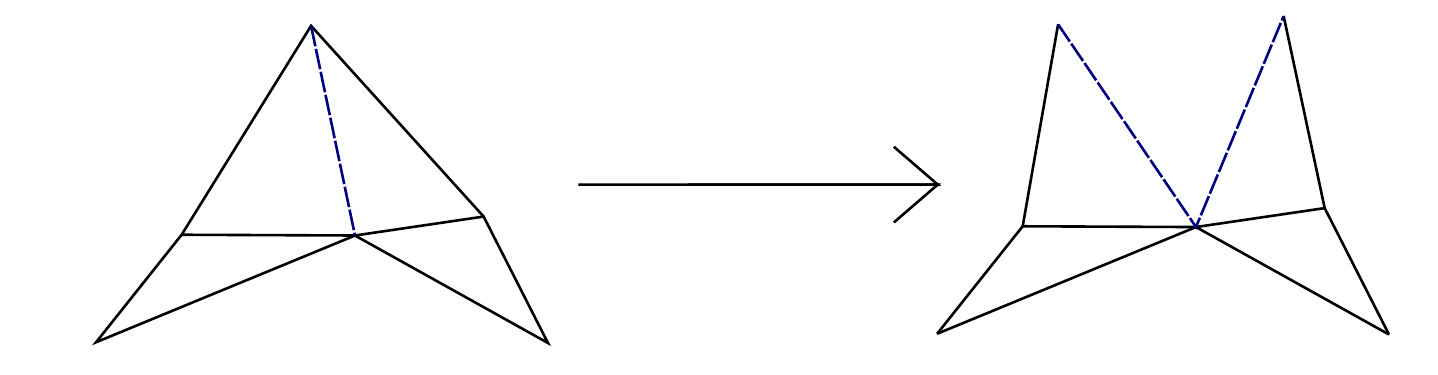}}%
  \end{picture}%
\endgroup%

%
%\item Thus suppose the link of a vertex $w' \in \partial D$ contains $u' \in \rho^{-1}(u)$ and $v' \in \rho^{-1}(v)$ and there is in edge between $u'$ and $v'$ in $D$. Then there is triangle with vertices $u'$, $v'$ and $w'$. 
%
%\input{TriangleExists.pdf_tex}

\end{itemize}

Combining all of the above we see that $D$ must look like the following picture.

%% Accompanies image file 'OptimalDisc.pdf' (pdf, eps, ps)
\begingroup%
  \makeatletter%
  \providecommand\color[2][]{%
    \errmessage{(Inkscape) Color is used for the text in Inkscape, but the package 'color.sty' is not loaded}%
    \renewcommand\color[2][]{}%
  }%
  \providecommand\transparent[1]{%
    \errmessage{(Inkscape) Transparency is used (non-zero) for the text in Inkscape, but the package 'transparent.sty' is not loaded}%
    \renewcommand\transparent[1]{}%
  }%
  \providecommand\rotatebox[2]{#2}%
  \newcommand*\fsize{\dimexpr\f@size pt\relax}%
  \newcommand*\lineheight[1]{\fontsize{\fsize}{#1\fsize}\selectfont}%
  \ifx\svgwidth\undefined%
    \setlength{\unitlength}{413.85826772bp}%
    \ifx\svgscale\undefined%
      \relax%
    \else%
      \setlength{\unitlength}{\unitlength * \real{\svgscale}}%
    \fi%
  \else%
    \setlength{\unitlength}{\svgwidth}%
  \fi%
  \global\let\svgwidth\undefined%
  \global\let\svgscale\undefined%
  \makeatother%
  \begin{picture}(1,0.26027397)%
    \lineheight{1}%
    \setlength\tabcolsep{0pt}%
    \put(0,0){\includegraphics[width=\unitlength,page=1]{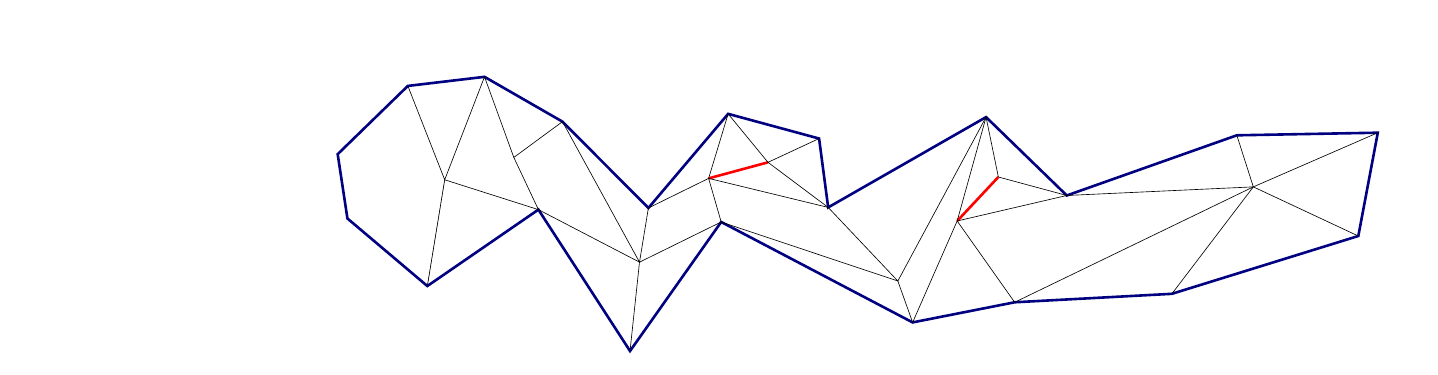}}%
    \put(0.72902972,0.08734742){\color[rgb]{1,0,0}\makebox(0,0)[t]{\lineheight{1.25}\smash{\begin{tabular}[t]{c}$\rho^{-1}(e)$\end{tabular}}}}%
    \put(0.90957129,0.17686856){\color[rgb]{0,0,0.50196078}\makebox(0,0)[t]{\lineheight{1.25}\smash{\begin{tabular}[t]{c}$\partial D$\end{tabular}}}}%
    \put(0,0){\includegraphics[width=\unitlength,page=2]{OptimalDisc.pdf}}%
    \put(0.12823179,0.14733904){\color[rgb]{0,0,0.50196078}\makebox(0,0)[t]{\lineheight{1.25}\smash{\begin{tabular}[t]{c}$\rho^{-1}(p_0) = \rho^{-1}(a)$\end{tabular}}}}%
    \put(0.26063548,0.03246407){\color[rgb]{0,0,0.50196078}\makebox(0,0)[t]{\lineheight{1.25}\smash{\begin{tabular}[t]{c}$\rho^{-1}(p_n) = \rho^{-1}(b)$\end{tabular}}}}%
    \put(0.19452974,0.09590628){\color[rgb]{0,0.50196078,0}\makebox(0,0)[t]{\lineheight{1.25}\smash{\begin{tabular}[t]{c}$\rho^{-1}(u)$\end{tabular}}}}%
    \put(0.34803876,0.22130396){\color[rgb]{0,0,0.50196078}\makebox(0,0)[t]{\lineheight{1.25}\smash{\begin{tabular}[t]{c}$\rho^{-1}(p_1)$\end{tabular}}}}%
    \put(0.18635035,0.08078032){\color[rgb]{0,0,0}\makebox(0,0)[lt]{\begin{minipage}{0.0329105\unitlength}\centering \end{minipage}}}%
    \put(0.03662987,0.3576232){\color[rgb]{0,0,0}\makebox(0,0)[lt]{\begin{minipage}{0.93116981\unitlength}\centering \end{minipage}}}%
    \put(0,0){\includegraphics[width=\unitlength,page=3]{OptimalDisc.pdf}}%
  \end{picture}%
\endgroup%

It follows that $\gamma$ can be decomposed into locally injective subpaths $\gamma_i$ between $p_{i-1}$ and $p_i$ with the following properties. ($1 \leq i \leq n$)
\begin{itemize}
\item Each $\gamma_i$ is contained in the link of either $u$ or $v$.
\item For each $i$ there is a triangle $t_i$ with vertices $p_i$, $u$ and $v$.
\end{itemize}
Thus for each $i$ there is a cone with central point either $u$ or $v$ containing both $t_{i-1}$ and $t_i$. Thus by \thref{res:make_cones_simple_again} either $t_{i-1} = t_i$ or there is a simple cone containing both $t_{i-1}$ and $t_i$. Thus by the assumption in the statement case we get $t_{i-1} \sim t_i$ for all $i$ and so $t = t_0 \sim t_n = t'$. $\square$
\end{proof}

Let $\sigma_{n,m}$ be the surjective map induced by $\tau_{n,m}$ on the equivalence classes of $\sim_n$ and $\sim_m$. Our goal shall be to show that, far enough down the hierarchy, this $\sigma_{n,m}$ is always a bijection. Then we will show that the corresponding subcomplexes are themselves rigid which will allow us to prove \thref{res:main}.

\begin{prop}\thlabel{res:same_class_implies_stable}
For $n > N_{\Delta}$ every pair which is contained in the subcomplex associated to $\sim_n$ is a stable pair. 
\end{prop}

\begin{proof}
First note that a subcomplex $Y$ corresponding to an equivalence class must be cutpoint free as any two triangles it contains must be joined by a sequence of stable pairs. This means that the intersection of $Y$ and any track from the resolution $\rho$ is either trivial or parallel to a vertex of $Y$. $\square$

%% Accompanies image file 'TracksParallelToVertices.pdf' (pdf, eps, ps)
\begingroup%
  \makeatletter%
  \providecommand\color[2][]{%
    \errmessage{(Inkscape) Color is used for the text in Inkscape, but the package 'color.sty' is not loaded}%
    \renewcommand\color[2][]{}%
  }%
  \providecommand\transparent[1]{%
    \errmessage{(Inkscape) Transparency is used (non-zero) for the text in Inkscape, but the package 'transparent.sty' is not loaded}%
    \renewcommand\transparent[1]{}%
  }%
  \providecommand\rotatebox[2]{#2}%
  \newcommand*\fsize{\dimexpr\f@size pt\relax}%
  \newcommand*\lineheight[1]{\fontsize{\fsize}{#1\fsize}\selectfont}%
  \ifx\svgwidth\undefined%
    \setlength{\unitlength}{413.85826772bp}%
    \ifx\svgscale\undefined%
      \relax%
    \else%
      \setlength{\unitlength}{\unitlength * \real{\svgscale}}%
    \fi%
  \else%
    \setlength{\unitlength}{\svgwidth}%
  \fi%
  \global\let\svgwidth\undefined%
  \global\let\svgscale\undefined%
  \makeatother%
  \begin{picture}(1,0.26027397)%
    \lineheight{1}%
    \setlength\tabcolsep{0pt}%
    \put(0.03662987,0.3576232){\color[rgb]{0,0,0}\makebox(0,0)[lt]{\begin{minipage}{0.93116981\unitlength}\centering \end{minipage}}}%
    \put(0,0){\includegraphics[width=\unitlength,page=1]{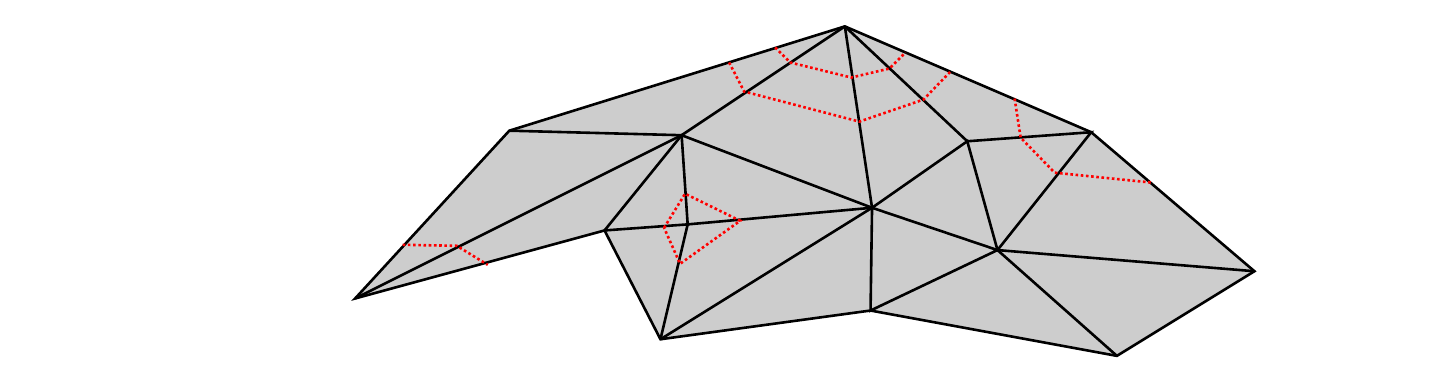}}%
  \end{picture}%
\endgroup%

\end{proof}

\begin{lemma}\thlabel{res:components_eventually_stable}
There exists $N' \geq N_{\Delta}$ such that $\sigma_{n,m}$ is a bijection whenever $n,m \geq N'$.
\end{lemma}

\begin{proof}
We shall proceed by proving the following three claims about the structure of the classes of $\sim_n$.

\begin{proof}[Claim 1]
There is some $N_1 \geq N_{\Delta}$ such that number of orbits of equivalence classes of $\sim_n$ and $\sim_m$ are equal whenever $n,m \geq N_1$.
\end{proof}

\begin{proof}[Proof of Claim 1]
Since $\tau_{n,m}$ induces a surjective equivariant map on the equivalence classes we see that the number of $G$-orbits of $G$ classes is non-increasing, hence must be eventually constant. $\square$
\end{proof}

Let $Y^1_n, \cdots , Y^{J}_n$ be the associated subcomplexes to a set of representatives for the orbits of $\sim_{n}$ and w.l.o.g. we can assume that $Y^j_n$ maps into $Y^{j}_{n+1}$.

\begin{proof}[Claim 2]
There is some $N_2 \geq N_1$ such that the number of orbits of edges in each $Y^j_n$ is constant for $n \geq N_2$.
\end{proof}

\begin{proof}[Proof of Claim 2]
Recall from \thref{res:edge_passdown} that for any given triangle there is a natural correspondence between its edges at any given level. Thus the only way to increase the number of edges is if two triangles are adjacent on one level but then not on a later one. This contradicts \thref{res:same_class_implies_stable}. $\square$
\end{proof}

The proof of Claim~2 also means we can meaningfully talk about the image of an edge under $\tau_{n,m}$ as long as we restrict our attention to a single equivalence class.

\begin{proof}[Claim 3]\cite[Lemma 8.2]{LouderTouikan_strong}
For $n > N_2$ if $\sigma_{n,n+1}$ isn't a bijection then there is some $j$ and an edge $e \subset Y^j_n$ such that 
\begin{equation*}
\stab^{+}_{Y^j_n}(e) \lneq \stab^{+}_{Y^j_{n+1}}(\tau_{n,n+1}(e))
\end{equation*}
\end{proof}

\begin{proof}[Proof of Claim 3] 
Claim 1 implies that $Y^j_n$ must join onto a conjugate of itself under $\tau_{n,n+1}$. Claim 2 implies that we have an edge $e \subset Y^j_n$ and a $g \in G \:\backslash\: \stab(Y^j_{n+1})$ such that $\tau(e) = \tau(ge)$. We thus have $g \in \stab^{+}_{Y^j_{n+1}}(\tau_{n,n+1}(e)) \:\backslash\: \stab^{+}_{Y^j_n}(e)$. $\square$
\end{proof}

We are now ready to show that $\sigma_{n,n+1}$ is a bijection for all sufficiently large $n$. Suppose this isn't the case; then since there are only finitely many orbits of edges in each $Y^j_n$ Claim~3 now implies that there is some $j$ and some subsequence $\left\lbrace n_{i_k} \right\rbrace$ of $\left\lbrace n_i \right\rbrace$ and some edge $e \in Y^j_{N_2}$ such that
\begin{equation*}
\stab^{+}_{Y^j_{n_{i_1}}}(e_{n_{i_1}}) \:\: \lneq \:\: \stab^{+}_{Y^j_{n_{i_2}}}(e_{n_{i_2}}) \:\: \lneq \:\: \stab^{+}_{Y^j_{n_{i_3}}}(e_{n_{i_3}}) \:\: \lneq \:\: \cdots
\end{equation*}
where $e_n = \tau_{N_2,n}(e) \in Y^j_{n}$. However this is exactly the situation the ACC says cannot happen. $\square$
\end{proof}

Before proving \thref{res:main} we need one final statement about the rigidity of the $Y^j_n$. This essentially says that eventually these complexes look identical at every level.

\begin{lemma}\thlabel{res:pair_pullback}
There is some $N'' \geq N'$ with the following property. Suppose $n,m \geq N''$ and $t, t'$ are triangles at depth $n$ where $(\tau_{n,m}(t),\tau_{n,m}(t'))$ is a stable pair at depth $m$. Then $(t,t')$ is also a (stable) pair at depth $n$. In other words $\tau_{n,m}$ induces a bijection on the set of stable pairs for $n,m \geq N''$.
\end{lemma}

\begin{proof}
Let $e, e'$ be the respective edges of $t, t'$ which get mapped to the common edge of the pair $(\tau_{n,m}(t),\tau_{n,m}(t'))$. Since $n \geq N'$ we must have $t$ and $t'$ in the same equivalence class of $\sim_n$; call the corresponding subcomplex $Y$. We also must have $e' = ge$ for some $g$ which stabilises $Y$ as $N' > N_2$. The same argument as the proof of Claim~3 in the proof of \thref{res:components_eventually_stable} implies that either $e' = e$ or 
$\stab^{+}_{Y}(e) \lneq \stab^{+}_{\tau_{n,m}(Y)}(\tau_{n,m}(e))$. As in the proof of \thref{res:components_eventually_stable} the ACC says this latter case can only occur finitely many times. $\square$
\end{proof}

\begin{proof}[Proof of \thref{res:main}]
In light of \thref{res:cone_criterion} and \thref{res:triangle_passdown} it suffices to show that every simple cone in a complex of depth at least $N''$ is contained in an equivalence class. We define the \emph{push-forward} of a cone as follows. Let $c$ be the central vertex of a cone $C$. If the resolution $\rho$ induces track(s) on $X$ whose intersection with $C$ is homeomorphic to a circle enclosing $c$ then we let $s$ be the outermost such track. Otherwise set $s=c$. We now define the push-forward of $C$ to be the union of the image of the triangles in $C$ which are in the same component as the image of $s$.

%% Accompanies image file 'ConePushForward.pdf' (pdf, eps, ps)
\begingroup%
  \makeatletter%
  \providecommand\color[2][]{%
    \errmessage{(Inkscape) Color is used for the text in Inkscape, but the package 'color.sty' is not loaded}%
    \renewcommand\color[2][]{}%
  }%
  \providecommand\transparent[1]{%
    \errmessage{(Inkscape) Transparency is used (non-zero) for the text in Inkscape, but the package 'transparent.sty' is not loaded}%
    \renewcommand\transparent[1]{}%
  }%
  \providecommand\rotatebox[2]{#2}%
  \newcommand*\fsize{\dimexpr\f@size pt\relax}%
  \newcommand*\lineheight[1]{\fontsize{\fsize}{#1\fsize}\selectfont}%
  \ifx\svgwidth\undefined%
    \setlength{\unitlength}{413.85826772bp}%
    \ifx\svgscale\undefined%
      \relax%
    \else%
      \setlength{\unitlength}{\unitlength * \real{\svgscale}}%
    \fi%
  \else%
    \setlength{\unitlength}{\svgwidth}%
  \fi%
  \global\let\svgwidth\undefined%
  \global\let\svgscale\undefined%
  \makeatother%
  \begin{picture}(1,0.26027397)%
    \lineheight{1}%
    \setlength\tabcolsep{0pt}%
    \put(0.03662987,0.3576232){\color[rgb]{0,0,0}\makebox(0,0)[lt]{\begin{minipage}{0.93116981\unitlength}\centering \end{minipage}}}%
    \put(0,0){\includegraphics[width=\unitlength,page=1]{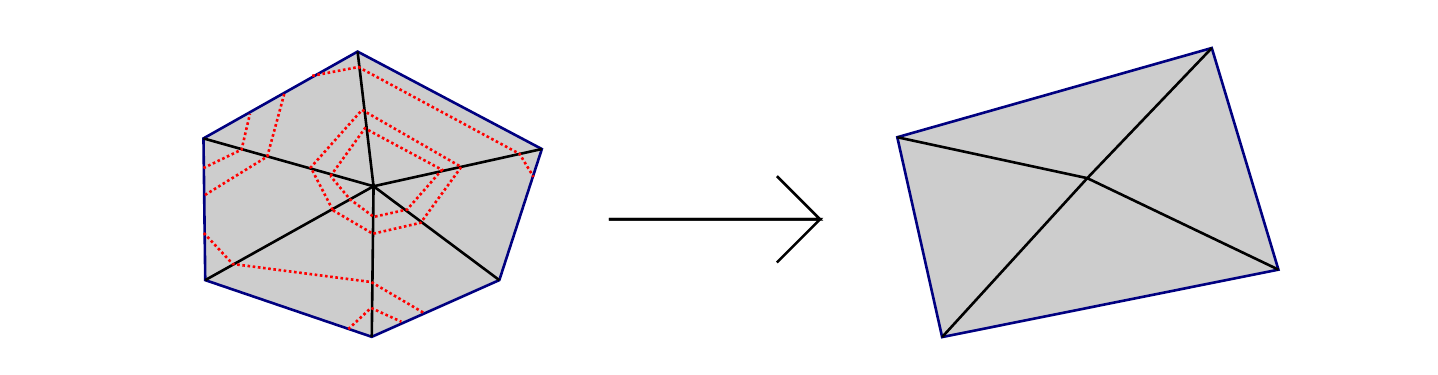}}%
    \put(0.26920358,0.1383065){\color[rgb]{0,0,0}\makebox(0,0)[t]{\lineheight{1.25}\smash{\begin{tabular}[t]{c}$c$\end{tabular}}}}%
    \put(0.32156247,0.11904238){\color[rgb]{1,0,0}\makebox(0,0)[t]{\lineheight{1.25}\smash{\begin{tabular}[t]{c}$s$\end{tabular}}}}%
    \put(0.15361893,0.18177419){\color[rgb]{0,0,0.50196078}\makebox(0,0)[t]{\lineheight{1.25}\smash{\begin{tabular}[t]{c}$\partial D$\end{tabular}}}}%
    \put(0.65548021,0.1838266){\color[rgb]{0,0,0.50196078}\makebox(0,0)[t]{\lineheight{1.25}\smash{\begin{tabular}[t]{c}$\partial D'$\end{tabular}}}}%
    \put(0,0){\includegraphics[width=\unitlength,page=2]{ConePushForward.pdf}}%
    \put(0.75426405,0.14373988){\color[rgb]{1,0,0}\makebox(0,0)[t]{\lineheight{1.25}\smash{\begin{tabular}[t]{c}$s$\end{tabular}}}}%
  \end{picture}%
\endgroup%

Let $C$ be a simple cone at depth $n \geq N''$. Apply push-forwards to $C$ until we reach a cone with minimal circumference; call this new cone $C'$. Observe that $C'$ is made of consecutive stable pairs and so is contained in an equivalence class. If $C'$ has the same circumference as $C$ then we are done, so assume that the circumference of $C'$ is strictly smaller than that of $C$. In this case we see $C'$ must contain a (stable) pair of adjacent triangles that weren't adjacent in $C$; however this contradicts \thref{res:pair_pullback}. $\square$
\end{proof}

\end{document}